\documentclass[12pt, a4paper, reqno]{amsart}
\usepackage{amsmath}
\usepackage{amsfonts}
\usepackage{amssymb}
\usepackage{color}
\usepackage{comment,graphicx}
\usepackage[T1]{fontenc}
\usepackage{url}
\usepackage{ae}
\usepackage{mathtools}
\usepackage{esint}
\usepackage{latexsym}
\usepackage{amscd}

\def\phi{\varphi}
\def\rho{\varrho}
\def\epsilon{\varepsilon}
\numberwithin{equation}{section}
\theoremstyle{plain}
\newtheorem{theorem}[equation]{Theorem}
\newtheorem{lemma}[equation]{Lemma}
\newtheorem{proposition}[equation]{Proposition}
\newtheorem{corollary}[equation]{Corollary}
\theoremstyle{definition}
\newtheorem{definition}[equation]{Definition}

\theoremstyle{remark}
\newtheorem{remark}[equation]{Remark}
\newtheorem{example}[equation]{Example}

\renewcommand{\leq}{\leqslant}
\renewcommand{\geq}{\geqslant}

\begin{document}
\title[Weigthed function spaces]{Complex interpolation of function spaces
with general weights}
\author[D. Drihem]{Douadi Drihem}
\address{Douadi Drihem\\
M'sila University\\
Department of Mathematics\\
Laboratory of Functional Analysis and Geometry of Spaces\\
M'sila 28000, Algeria.}
\email{douadidr@yahoo.fr, douadi.drihem@univ-msila.dz}
\thanks{ }
\date{\today }

\begin{abstract}
In this paper, we present the complex interpolation of Besov and
Triebel-Lizorkin spaces with generalized smoothness. In some particular
cases these function spaces are just weighted Besov and Triebel-Lizorkin
spaces. An application, we obtain the complex interpolation between the
weighted Triebel-Lizorkin spaces $\dot{F}_{p_{0},q_{0}}^{s_{0}}(\omega _{0})$
and $\dot{F}_{\infty ,q_{1}}^{s_{1}}(\omega _{1})\ $with suitable
assumptions on the parameters $\ s_{0},s_{1},p_{0},$ $q_{0}$ and$\ q_{1}$,
and the pair of weights $(\omega _{0},\omega _{1})$.
\end{abstract}

\maketitle

\section{Introduction}

The theory of complex interpolation had a remarkable development due to its
usefulness in applications in mathematical analysis. For general literature
on complex interpolation we refer to \cite{BL76}, \cite{T78} and references
therein. Let us recall briefly the results of complex interpolation of some
know function spaces. For Lebesgue space we have%
\begin{equation*}
\lbrack L_{p_{0}}(\mathbb{R}^{n}),L_{p_{1}}(\mathbb{R}^{n})]_{\theta }=L_{p}(%
\mathbb{R}^{n}),
\end{equation*}%
see \cite[Theorem 5.1.1]{BL76}, where%
\begin{equation}
0<\theta <1,\quad 1\leq p_{0},p_{1}<\infty ,\quad \frac{1}{p}:=\frac{%
1-\theta }{p_{0}}+\frac{\theta }{p_{1}}.  \label{ass-Lebesgue}
\end{equation}%
Let $L_{p}(\mathbb{R}^{n},\omega )$ denote the weighted Lebesgue space with
weight $\omega $, see below. The weighted version is given as follows: 
\begin{equation}
\lbrack L_{p_{0}}(\mathbb{R}^{n},\omega _{0}),L_{p_{1}}(\mathbb{R}%
^{n},\omega _{1})]_{\theta }=L_{p}(\mathbb{R}^{n},\omega _{0}^{\theta
}\omega _{1}^{1-\theta }),  \label{Lebesgue}
\end{equation}%
see \cite[Theorem 5.5.3]{BL76} and \cite[Theorem 1.18.5]{T78} with the same
assumptions \eqref{ass-Lebesgue}. Clearly all the above results are given
for Banach case, but in \cite[Lemma 3.4]{SSV13} the authors gave a
generalization of \eqref{Lebesgue} to the case $0<p_{0},p_{1}<\infty $. For
Sobolev spaces%
\begin{equation*}
\lbrack W_{p_{0}}^{m_{0}}(\mathbb{R}^{n}),W_{p_{1}}^{m_{1}}(\mathbb{R}%
^{n})]_{\theta }=W_{p}^{m}(\mathbb{R}^{n}),
\end{equation*}%
with the same assumptions \eqref{ass-Lebesgue}, $1<p_{0},p_{1}<\infty $ and $%
m_{0},m_{1}\in \mathbb{N},m=(1-\theta )m_{0}+\theta m_{1}$, see \cite[Remark
2.4.2/2 ]{T78}. The extension of the above results to generalized scale of
function spaces are given in \cite{BL76} and \cite{T78}. For convenience of
the reader we recall some know results on Besov and Triebel-Lizorkin spaces
and it is known that they cover many well-known classical function spaces
such as H\"{o}lder-Zygmund spaces, Hardy spaces and Sobolev spaces. For more
details one can refer to Triebel's books \cite{T1} and \cite{T2}, which are
denoted by $B_{p,q}^{s}(\mathbb{R}^{n},\omega )$ and $B_{p,q}^{s}(\mathbb{R}%
^{n},\omega )$, respectively, see Section 3. Let 
\begin{equation}
0<\theta <1,\quad 0<p_{0},p_{1},q_{0},q_{1}\leq \infty ,\quad \frac{1}{p}:=%
\frac{1-\theta }{p_{0}}+\frac{\theta }{p_{1}},\quad \frac{1}{q}:=\frac{%
1-\theta }{q_{0}}+\frac{\theta }{q_{1}}.  \label{ass-besov1}
\end{equation}%
Then we have%
\begin{equation*}
\lbrack B_{p_{0},q_{0}}^{s_{0}}(\mathbb{R}^{n}),B_{p_{1},q_{1}}^{s_{1}}(%
\mathbb{R}^{n})]_{\theta }=B_{p,q}^{s}(\mathbb{R}^{n}),
\end{equation*}%
if 
\begin{equation}
s_{0},s_{1}\in \mathbb{R},s=(1-\theta )s_{0}+\theta
s_{1},1<p_{0},p_{1}<\infty ,1\leq q_{0}<\infty ,1\leq q_{1}\leq \infty ,
\label{ass-besov2}
\end{equation}%
see \cite[Theorem 2.4.1 ]{T78}. We mention that in some cases complex
interpolation of pairs of Besov spaces does not result in a Besov space.
More precisely%
\begin{equation}
\lbrack B_{p,\infty }^{s_{0}}(\mathbb{R}^{n}),B_{p,\infty }^{s_{1}}(\mathbb{R%
}^{n})]_{\theta }=\mathring{B}_{p,\infty }^{s}(\mathbb{R}^{n}),
\label{contradition}
\end{equation}%
see again \cite[Theorem 2.4.1 ]{T78}, with $s_{0}\neq s_{1}\in \mathbb{R}%
,s=(1-\theta )s_{0}+\theta s_{1},1<p<\infty $, where $\mathring{B}_{p,q}^{s}(%
\mathbb{R}^{n})$ is the closure of the set of test functions in $B_{p,\infty
}^{s}(\mathbb{R}^{n})$. The result \eqref{contradition} is in contradiction
with \cite[Theorem 6.4.5]{BL76} in the case of $s_{0}\neq s_{1}$ and $%
q_{0}=q_{1}=\infty $, because of $\mathring{B}_{p,\infty }^{s}(\mathbb{R}%
^{n})$ is smaller that $B_{p,\infty }^{s}(\mathbb{R}^{n})$. For
Triebel-Lizorkin spaces we have%
\begin{equation*}
\lbrack F_{p_{0},q_{0}}^{s_{0}}(\mathbb{R}^{n}),F_{p_{1},q_{1}}^{s_{1}}(%
\mathbb{R}^{n})]_{\theta }=F_{p,q}^{s}(\mathbb{R}^{n}),
\end{equation*}%
with the same assumptions \eqref{ass-besov1} and \eqref{ass-besov2} but with 
$1<q_{0}<\infty ,1<q_{1}<\infty $, see \cite[Theorem 2.4.2/1 ]{T78}. The
quasi-Banach version of complex interpolation of Besov and Triebel-Lizorkin
spaces can be found in \cite{FJ90} and \cite{KMM07}. For weighted Besov and
Lizorkin-Triebel spaces Bownik \cite{M08} and Wojciechowska \cite{Wo12}
studied the problem%
\begin{equation}
\lbrack F_{p_{0},q_{0}}^{s_{0}}(\mathbb{R}^{n},\omega
_{0}),F_{p_{1},q_{1}}^{s_{1}}(\mathbb{R}^{n},\omega _{1})]_{\theta }.
\label{SSY}
\end{equation}%
But both authors only deal with the case $\omega _{0}=\omega _{1}=\omega $.
The general problem, $\omega _{0}\neq \omega _{1}$, was completed in \cite%
{SSV13} but with $p_{0}<\infty $ and $p_{1}<\infty $ and the weights $\omega
_{0},\omega _{1}$ are local Muckenhoupt weights in the sense of Rychkov \cite%
{Ry01}.

In this direction we present the complex interpolation of function spaces of
general weights, were introduced in \cite{D20} and \cite{D20.1}, that based
on Tyulenev class given in \cite{Ty15}, \cite{Ty-151} and \cite{Ty-N-L}. An
application, we study \eqref{SSY}, when $p_{1}=\infty $ and $\omega _{0}\neq
\omega _{1}$.

These type of spaces of generalized smoothness are have been introduced by
several authors. We refer, for instance, to Bownik \cite{M07}, Cobos and
Fernandez \cite{CF88}, Goldman \cite{Go79} and \cite{Go83}, and Kalyabin 
\cite{Ka83}; see also\ Besov \cite{B03} and \cite{B05}, and Kalyabin and
Lizorkin \cite{Kl87}.

The theory of these spaces had a remarkable development in part due to its
usefulness in applications. For instance, they appear in the study of trace
spaces on fractals, see Edmunds and Triebel \cite{ET96} and \cite{ET99},
were they introduced the spaces $B_{p,q}^{s,\Psi }$, where $\Psi $ is a
so-called admissible function, typically of log-type near $0$. For a
complete treatment of these spaces we refer the reader the work of Moura 
\cite{Mo01}. More general function spaces of generalized smoothness can be
found in Farkas and Leopold \cite{FL06}, and reference therein.

The paper is organized as follows. First we recall some basic facts on the
Muckenhoupt classes and the weighted class of Tyulenev. Also we give some
key technical lemmas needed in the proofs of the main statements. In Section
3, we present some properties of $\dot{B}_{p,q}(\mathbb{R}^{n},\{t_{k}\})$
and $\dot{F}_{p,q}(\mathbb{R}^{n},\{t_{k}\})$ spaces. In Section 4 we shall
apply a method which has been used by \cite{SSV13}. First we shall calculate
the Calder\'{o}n products of associated sequence spaces. Then, from an
abstract theory on the relation between the complex interpolation and the
Calder\'{o}n product of Banach lattices obtained by Calder\'{o}n \cite{Ca64}%
, we deduce the complex interpolation theorems of these sequence spaces.
Finally, the desired complex interpolation theorem is lifted by the $\varphi 
$-transform characterization in the sense of Frazier and Jawerth.\vskip5pt

We will adopt the following convention throughout this paper. As usual, we
denote by $\mathbb{R}^{n}$ the $n$-dimensional real Euclidean space, while $%
\mathbb{N}$ the collection of all natural numbers and $\mathbb{N}_{0}=%
\mathbb{N}\cup \{0\}$. The letter $\mathbb{Z}$ stands for the set of all
integer numbers.\ The expression $f\lesssim g$ means that $f\leq c\,g$ for
some independent constant $c$ (and non-negative functions $f$ and $g$), and $%
f\approx g$ means $f\lesssim g\lesssim f$. \vskip5pt

By supp $f$ we denote the support of the function $f$, i.e., the closure of
its non-zero set. If $E\subset {\mathbb{R}^{n}}$ is a measurable set, then $%
|E|$ stands for the (Lebesgue) measure of $E$ and $\chi _{E}$ denotes its
characteristic function. By $c$ we denote generic positive constants, which
may have different values at different occurrences. \vskip5pt

A weight is a nonnegative locally integrable function on $\mathbb{R}^{n}$
that takes values in $(0,\infty )$ almost everywhere. For measurable set $%
E\subset \mathbb{R}^{n}$ and a weight $\gamma $, $\gamma (E)$ denotes 
\begin{equation*}
\int_{E}\gamma (x)dx.
\end{equation*}%
Given a measurable set $E\subset \mathbb{R}^{n}$ and $0<p\leq \infty $, we
denote by $L_{p}(E)$ the space of all functions $f:E\rightarrow \mathbb{C}$
equipped with the quasi-norm 
\begin{equation*}
\big\|f|L_{p}(E)\big\|:=\Big(\int_{E}\left\vert f(x)\right\vert ^{p}dx\Big)^{%
\frac{1}{p}}<\infty ,
\end{equation*}%
with $0<p<\infty $ and%
\begin{equation*}
\big\|f|L_{\infty }(E)\big\|:=\underset{x\in E}{\text{ess-sup}}\left\vert
f(x)\right\vert <\infty .
\end{equation*}%
For a function $f$ in $L_{1}^{\mathrm{loc}}$, we set 
\begin{equation*}
M_{A}(f):=\frac{1}{|A|}\int_{A}\left\vert f(x)\right\vert dx
\end{equation*}%
for any $A\subset \mathbb{R}^{n}$. Furthermore, we put%
\begin{equation*}
M_{A,p}(f):=\Big(\frac{1}{|A|}\int_{A}\left\vert f(x)\right\vert ^{p}dx\Big)%
^{\frac{1}{p}},
\end{equation*}%
with $0<p<\infty $. Further, given a measurable set $E\subset \mathbb{R}^{n}$
and a weight $\gamma $, we denote the space of all functions $f:\mathbb{R}%
^{n}\rightarrow \mathbb{C}$ with finite quasi-norm 
\begin{equation*}
\big\|f|L_{p}(\mathbb{R}^{n},\gamma )\big\|=\big\|f\gamma |L_{p}(\mathbb{R}%
^{n})\big\|
\end{equation*}%
by $L_{p}(\mathbb{R}^{n},\gamma )$.

If $1\leq p\leq \infty $ and $\frac{1}{p}+\frac{1}{p^{\prime }}=1$, then $%
p^{\prime }$ is called the conjugate exponent of $p$.

The symbol $\mathcal{S}(\mathbb{R}^{n})$ is used in place of the set of all
Schwartz functions on $\mathbb{R}^{n}$. In what follows, $Q$ will denote an
cube in the space $\mathbb{R}^{n}$\ with sides parallel to the coordinate
axes and $l(Q)$\ will denote the side length of the cube $Q$. For $v\in 
\mathbb{Z}$ and $m\in \mathbb{Z}^{n}$, denote by $Q_{v,m}$ the dyadic cube,%
\begin{equation*}
Q_{v,m}:=2^{-v}([0,1)^{n}+m).
\end{equation*}%
For the collection of all such cubes we use $\mathcal{Q}:=\{Q_{v,m}:v\in 
\mathbb{Z},m\in \mathbb{Z}^{n}\}$.

\section{Basic tools}

In this section we present some useful results.

\subsection{Muckenhoupt weights}

The purpose of this subsection is to review some known properties of\
Muckenhoupt class.

\begin{definition}
Let $1<p<\infty $. We say that a weight $\gamma $ belongs to the Muckenhoupt
class $A_{p}(\mathbb{R}^{n})$ if there exists a constant $C>0$ such that for
every cube $Q$ the following inequality holds 
\begin{equation}
M_{Q}(\gamma )M_{Q,\frac{p^{\prime }}{p}}(\gamma ^{-1})\leq C.
\label{Ap-constant}
\end{equation}
\end{definition}

The smallest constant $C$ for which $\mathrm{\eqref{Ap-constant}}$ holds,
denoted by $A_{p}(\gamma )$. As an example, we can take%
\begin{equation*}
\gamma (x)=|x|^{\alpha },\quad \alpha \in \mathbb{R}.
\end{equation*}%
Then $\gamma \in A_{p}(\mathbb{R}^{n})$, $1<p<\infty $, if and only if $%
-n<\alpha <n(p-1)$.

For $p=1$ we rewrite the above definition in the following way.

\begin{definition}
We say that a weight $\gamma $ belongs to the Muckenhoupt class $A_{1}(%
\mathbb{R}^{n})$ if there exists a constant $C>0$ such that for every cube $%
Q $\ and for a.e.\ $y\in Q$ the following inequality holds 
\begin{equation}
M_{Q}(\gamma )\leq C\gamma (y).  \label{A1-constant}
\end{equation}
\end{definition}

The smallest constant $C$ for which $\mathrm{\eqref{A1-constant}}$ holds,
denoted by $A_{1}(\gamma )$. The above classes have been first studied by
Muckenhoupt\ \cite{Mu72} and use it to characterize the boundedness of the
Hardy-Littlewood maximal function on $L_{p}(\gamma )$, see the monographs 
\cite[Chapter 7]{Du01}, \cite{GR85}, \cite{L. Graf14} and \cite[Chapter 5]%
{St93}\ for a complete account on the theory of Muckenhoupt weights.

\begin{lemma}
\label{Ap-Property}Let $1\leq p<\infty $.\newline
$\mathrm{(i)}$ If $\gamma \in A_{p}(\mathbb{R}^{n})$, then for any $1\leq
p<q $, $\gamma \in A_{q}(\mathbb{R}^{n})$.\newline
$\mathrm{(ii)}$ Let $1<p<\infty $. $\gamma \in A_{p}(\mathbb{R}^{n})$ if and
only if $\gamma ^{1-p^{\prime }}\in A_{p^{\prime }}(\mathbb{R}^{n})$.\newline
$\mathrm{(iii)}$ Suppose that $\gamma \in A_{p}(\mathbb{R}^{n})$ for some $%
1<p<\infty $. Then there exist a $1<p_{1}<p<\infty $ such that $\gamma \in
A_{p_{1}}(\mathbb{R}^{n})$.
\end{lemma}

\subsection{The weight class $\dot{X}_{\protect\alpha ,\protect\sigma ,p}$}

Let $0<p\leq \infty $. A weight sequence $\{t_{k}\}$ is called $p$%
-admissible if $t_{k}\in L_{p}^{\mathrm{loc}}(\mathbb{R}^{n})$ for all $k\in 
\mathbb{Z}$. We mention here that 
\begin{equation*}
\int_{E}t_{k}^{p}(x)dx<c(k)
\end{equation*}%
for any $k\in \mathbb{Z}$ and any compact set $E\subset \mathbb{R}^{n}$. For
a $p$-admissible weight sequence $\{t_{k}\}$\ we set%
\begin{equation*}
t_{k,m}:=\big\|t_{k}|L_{p}(Q_{k,m})\big\|,\quad k\in \mathbb{Z},m\in \mathbb{%
Z}^{n}.
\end{equation*}

Tyulenev\ \cite{Ty14} introduced the following new weighted class\ and use
it to study Besov spaces of variable smoothness.

\begin{definition}
\label{Tyulenev-class}Let $\alpha _{1}$, $\alpha _{2}\in \mathbb{R}$, $%
p,\sigma _{1}$, $\sigma _{2}$ $\in (0,+\infty ]$, $\alpha =(\alpha
_{1},\alpha _{2})$ and let $\sigma =(\sigma _{1},\sigma _{2})$. We let $\dot{%
X}_{\alpha ,\sigma ,p}=\dot{X}_{\alpha ,\sigma ,p}(\mathbb{R}^{n})$ denote
the set of $p$-admissible weight sequences $\{t_{k}\}$ satisfying the
following conditions. There exist numbers $C_{1},C_{2}>0$ such that for any $%
k\leq j$\ and every cube $Q,$%
\begin{eqnarray}
M_{Q,p}(t_{k})M_{Q,\sigma _{1}}(t_{j}^{-1}) &\leq &C_{1}2^{\alpha _{1}(k-j)},
\label{Asum1} \\
(M_{Q,p}(t_{k}))^{-1}M_{Q,\sigma _{2}}(t_{j}) &\leq &C_{2}2^{\alpha
_{2}(j-k)}.  \label{Asum2}
\end{eqnarray}
\end{definition}

The constants $C_{1},C_{2}>0$ are independent of both the indexes $k$ and $j$%
.

\begin{remark}
$\mathrm{(i)}$\ We would like to mention that if $\{t_{k}\}$ satisfying $%
\mathrm{\eqref{Asum1}}$ with $\sigma _{1}=r\left( \frac{p}{r}\right)
^{\prime }$ and $0<r<p\leq \infty $, then $t_{k}^{p}\in A_{\frac{p}{r}}(%
\mathbb{R}^{n})$ for any $k\in \mathbb{Z}$ with\ $0<r<p<\infty $ and $%
t_{k}^{-r}\in A_{1}(\mathbb{R}^{n})$ for any $k\in \mathbb{Z}$\ with\ $%
p=\infty $.\newline
$\mathrm{(ii)}$ We say that $t_{k}\in A_{p}(\mathbb{R}^{n})$,\ $k\in \mathbb{%
Z}$, $1<p<\infty $ have the same Muckenhoupt constant if%
\begin{equation*}
A_{p}(t_{k})=c,\quad k\in \mathbb{Z},
\end{equation*}%
where $c$ is independent of $k$.\newline
$\mathrm{(iii)}$ Definition \ref{Tyulenev-class} is different from the one
used in \cite[Definition 2.1]{Ty14} and Definition 2.7 in \cite{Ty15},
because we used the boundedness of the maximal function on weighted Lebesgue
spaces.
\end{remark}

\begin{example}
\label{Example1}Let $0<r<p<\infty $, a weight $\omega ^{p}\in A_{\frac{p}{r}%
}(\mathbb{R}^{n})$ and $\{s_{k}\}=\{2^{ks}\omega ^{p}\}_{k\in \mathbb{Z}}$, $%
s\in \mathbb{R}$. Clearly, $\{s_{k}\}_{k\in \mathbb{Z}}$ lies in $\dot{X}%
_{\alpha ,\sigma ,p}$ for $\alpha _{1}=\alpha _{2}=s$, $\sigma =(r(\frac{p}{r%
})^{\prime },p)$.
\end{example}

\begin{remark}
\label{Tyulenev-class-properties}Let $0<\theta \leq p\leq \infty $. Let $%
\alpha _{1}$, $\alpha _{2}\in \mathbb{R}$, $\sigma _{1},\sigma _{2}\in
(0,+\infty ]$, $\sigma _{2}\geq p$, $\alpha =(\alpha _{1},\alpha _{2})$ and
let $\sigma =(\sigma _{1}=\theta \left( \frac{p}{\theta }\right) ^{\prime
},\sigma _{2})$. Let a $p$-admissible weight sequence $\{t_{k}\}\in \dot{X}%
_{\alpha ,\sigma ,p}$. Then%
\begin{equation*}
\alpha _{2}\geq \alpha _{1},
\end{equation*}%
see \cite{D20}.
\end{remark}

As usual, we put%
\begin{equation*}
\mathcal{M(}f)(x):=\sup_{Q}\frac{1}{|Q|}\int_{Q}\left\vert f(y)\right\vert
dy,\quad f\in L_{1}^{\mathrm{loc}},
\end{equation*}%
where the supremum\ is taken over all cubes with sides parallel to the axis
and $x\in Q$. Also we set 
\begin{equation*}
\mathcal{M}_{\sigma }(f):=\left( \mathcal{M(}\left\vert f\right\vert
^{\sigma })\right) ^{\frac{1}{\sigma }},\quad 0<\sigma <\infty .
\end{equation*}

Recall the vector-valued maximal inequality of Fefferman and Stein \cite%
{FeSt71}.

\begin{theorem}
\label{FS-inequality}Let $0<p<\infty ,0<q\leq \infty $ and $0<\sigma <\min
(p,q)$. Then%
\begin{equation}
\Big\|\Big(\sum\limits_{k=-\infty }^{\infty }\big(\mathcal{M}_{\sigma
}(f_{k})\big)^{q}\Big)^{\frac{1}{q}}|L_{p}(\mathbb{R}^{n})\Big\|\lesssim %
\Big\|\Big(\sum\limits_{k=-\infty }^{\infty }\left\vert f_{k}\right\vert ^{q}%
\Big)^{\frac{1}{q}}|L_{p}(\mathbb{R}^{n})\Big\|  \label{Fe-St71}
\end{equation}%
holds for all sequence of functions $\{f_{k}\}\in L_{p}(\ell _{q})$.
\end{theorem}

We state one of the main tools of this paper, see \cite{D20.1}.

\begin{lemma}
\label{key-estimate1}Let $1<\theta \leq p<\infty $\ and $1<q<\infty $. Let $%
\{t_{k}\}$\ be a $p$-admissible\ weight\ sequence\ such that $t_{k}^{p}\in
A_{\frac{p}{\theta }}(\mathbb{R}^{n})$, $k\in \mathbb{Z}$. Assume that $%
t_{k}^{p}$,\ $k\in \mathbb{Z}$ have the same Muckenhoupt constant, $A_{\frac{%
p}{\theta }}(t_{k}^{p})=c,k\in \mathbb{Z}$. Then%
\begin{equation*}
\Big\|\Big(\sum\limits_{k=-\infty }^{\infty }t_{k}^{q}\big(\mathcal{M(}f_{k})%
\big)^{q}\Big)^{\frac{1}{q}}|L_{p}(\mathbb{R}^{n})\Big\|\lesssim \Big\|\Big(%
\sum\limits_{k=-\infty }^{\infty }t_{k}^{q}\left\vert f_{k}\right\vert ^{q}%
\Big)^{\frac{1}{q}}|L_{p}(\mathbb{R}^{n})\Big\|
\end{equation*}%
holds for all sequences of functions $\{f_{k}\}\in L_{p}(\ell _{q})$. In
particular%
\begin{equation*}
\big\|\mathcal{M(}f_{k})|L_{p}(\mathbb{R}^{n},t_{k})\big\|\leq c\big\|%
f_{k}|L_{p}(\mathbb{R}^{n},t_{k})\big\|
\end{equation*}%
holds for all sequences $f_{k}\in L_{p}(\mathbb{R}^{n},t_{k})$, $k\in 
\mathbb{Z}$, where $c>0$ is independent of $k$.
\end{lemma}

\begin{remark}
\label{r-estimates}$\mathrm{(i)}$ We would like to mention that the result
of this lemma is true if we assume that $t_{k}^{p}\in A_{\frac{p}{\theta }}(%
\mathbb{R}^{n})$,\ $k\in \mathbb{Z}$, $1<p<\infty $ with%
\begin{equation*}
A_{\frac{p}{\theta }}(t_{k}^{p})\leq c,\quad k\in \mathbb{Z},
\end{equation*}%
where $c>0$ independent of $k$.$\newline
\mathrm{(ii)}$ A proof of this result\ for $t_{k}^{p}=\omega $, $k\in 
\mathbb{Z}$ may be found in \cite{AJ80} and \cite{Kok78}.$\newline
\mathrm{(iii)}$ In view of Lemma \ref{Ap-Property}/(iii) we can assume that $%
t_{k}^{p}\in A_{p}(\mathbb{R}^{n})$,\ $k\in \mathbb{Z}$, $1<p<\infty $ with%
\begin{equation*}
A_{p}(t_{k}^{p})\leq c,\quad k\in \mathbb{Z},
\end{equation*}%
where $c>0$ independent of $k$.
\end{remark}

\section{Function spaces}

In this section we\ present the Fourier analytical definition of Besov and
Triebel-Lizorkin spaces of variable smoothness and recall their basic
properties. Select a pair of Schwartz functions $\varphi $ and $\psi $
satisfy%
\begin{equation}
\text{supp}\mathcal{F}\varphi ,\mathcal{F}\psi \subset \big\{\xi :\frac{1}{2}%
\leq |\xi |\leq 2\big\},  \label{Ass1}
\end{equation}%
\begin{equation}
|\mathcal{F}\varphi (\xi )|,|\mathcal{F}\psi (\xi )|\geq c\quad \text{if}%
\quad \frac{3}{5}\leq |\xi |\leq \frac{5}{3}  \label{Ass2}
\end{equation}%
and 
\begin{equation}
\sum_{k=-\infty }^{\infty }\overline{\mathcal{F}\varphi (2^{-k}\xi )}%
\mathcal{F}\psi (2^{-k}\xi )=1\quad \text{if}\quad \xi \neq 0,  \label{Ass3}
\end{equation}%
where $c>0$. Throughout the paper, for all $k$ $\in \mathbb{Z}$ and $x\in 
\mathbb{R}^{n}$, we put $\varphi _{k}(x):=2^{kn}\varphi (2^{k}x)$ and $%
\tilde{\varphi}(x):=\overline{\varphi (-x)}$. Let $\varphi \in \mathcal{S}(%
\mathbb{R}^{n})$ be a function satisfying $\mathrm{\eqref{Ass1}}$-$\mathrm{%
\eqref{Ass2}}$. Recall that there exists a function $\psi \in \mathcal{S}(%
\mathbb{R}^{n})$ satisfying $\mathrm{\eqref{Ass1}}$-$\mathrm{\eqref{Ass3}}$,
see \cite[Lemma (6.9)]{FrJaWe01}. We set%
\begin{equation*}
\mathcal{S}_{\infty }(\mathbb{R}^{n}):=\Big\{\varphi \in \mathcal{S}(\mathbb{%
R}^{n}):\int_{\mathbb{R}^{n}}x^{\beta }\varphi (x)dx=0\text{ for all
multi-indices }\beta \in \mathbb{N}_{0}^{n}\Big\}.
\end{equation*}%
Let $\mathcal{S}_{\infty }^{\prime }(\mathbb{R}^{n})$ be the topological
dual of $\mathcal{S}_{\infty }(\mathbb{R}^{n})$, namely, the set of all
continuous linear functionals on $\mathcal{S}_{\infty }(\mathbb{R}^{n})$.

Now, we define the spaces under consideration.

\begin{definition}
\label{B-F-def}Let $0<p\leq \infty $ and $0<q\leq \infty $. Let $\{t_{k}\}$
be a $p$-admissible weight sequence, and $\varphi \in \mathcal{S}(\mathbb{R}%
^{n})$\ satisfy $\mathrm{\eqref{Ass1}}$ and $\mathrm{\eqref{Ass2}}$.\newline
$\mathrm{(i)}$ The Besov space $\dot{B}_{p,q}(\mathbb{R}^{n},\{t_{k}\})$\ is
the collection of all $f\in \mathcal{S}_{\infty }^{\prime }(\mathbb{R}^{n})$%
\ such that 
\begin{equation*}
\big\|f|\dot{B}_{p,q}(\mathbb{R}^{n},\{t_{k}\})\big\|:=\Big(%
\sum\limits_{k=-\infty }^{\infty }\big\|t_{k}(\varphi _{k}\ast f)|L_{p}(%
\mathbb{R}^{n})\big\|^{q}\Big)^{\frac{1}{q}}<\infty
\end{equation*}%
with the usual modifications if $q=\infty $.\newline
$\mathrm{(ii)}$ Let $0<p<\infty $. The Triebel-Lizorkin space $\dot{F}_{p,q}(%
\mathbb{R}^{n},\{t_{k}\})$\ is the collection of all $f\in \mathcal{S}%
_{\infty }^{\prime }(\mathbb{R}^{n})$\ such that 
\begin{equation*}
\big\|f|\dot{F}_{p,q}(\mathbb{R}^{n},\{t_{k}\})\big\|:=\Big\|\Big(%
\sum\limits_{k=-\infty }^{\infty }t_{k}^{q}|\varphi _{k}\ast f|^{q}\Big)^{%
\frac{1}{q}}|L_{p}(\mathbb{R}^{n})\Big\|<\infty
\end{equation*}%
with the usual modifications if $q=\infty $.
\end{definition}

\begin{remark}
Some properties of these function spaces, such as the $\varphi $-transform
characterization in the sense of Frazier and Jawerth, the smooth atomic and
molecular decomposition and the characterization of $\dot{B}_{p,q}(\mathbb{R}%
^{n},\{t_{k}\})$ spaces in terms of the difference relations are given in\ 
\cite{D20}, \cite{D20.1} and \cite{D20.2}.
\end{remark}

\begin{remark}
We would like to mention that the elements of the above spaces are not
distributions but equivalence classes of distributions$.$ We will use $\dot{A%
}_{p,q}(\mathbb{R}^{n},\{t_{k}\})$ to denote either $\dot{B}_{p,q}(\mathbb{R}%
^{n},\{t_{k}\})$ or $\dot{F}_{p,q}(\mathbb{R}^{n},\{t_{k}\})$.
\end{remark}

Using the system $\{\varphi _{k}\}_{k\in \mathbb{Z}}$ we can define the
quasi-norms%
\begin{equation*}
\big\|f|\dot{B}_{p,q}^{s}(\mathbb{R}^{n})\big\|:=\Big(\sum\limits_{k=-\infty
}^{\infty }2^{ksq}\big\|\varphi _{k}\ast f|L_{p}(\mathbb{R}^{n})\big\|^{q}%
\Big)^{\frac{1}{q}}
\end{equation*}%
and%
\begin{equation*}
\big\|f|\dot{F}_{p,q}^{s}(\mathbb{R}^{n})\big\|:=\Big\|\Big(%
\sum\limits_{k=-\infty }^{\infty }2^{ksq}|\varphi _{k}\ast f|^{q}\Big)^{%
\frac{1}{q}}|L_{p}(\mathbb{R}^{n})\Big\|
\end{equation*}%
for constants $s\in \mathbb{R}$ and $0<p,q\leq \infty $ with $0<p<\infty $
in the $\dot{F}$-case. The Besov space $\dot{B}_{p,q}^{s}(\mathbb{R}^{n})$\
consist of all distributions $f\in \mathcal{S}_{\infty }^{\prime }(\mathbb{R}%
^{n})$ for which 
\begin{equation*}
\big\|f|\dot{B}_{p,q}^{s}(\mathbb{R}^{n})\big\|<\infty .
\end{equation*}%
The Triebel-Lizorkin space $\dot{F}_{p,q}^{s}(\mathbb{R}^{n})$\ consist of
all distributions $f\in \mathcal{S}_{\infty }^{\prime }(\mathbb{R}^{n})$ for
which 
\begin{equation*}
\big\|f|\dot{F}_{p,q}^{s}(\mathbb{R}^{n})\big\|<\infty .
\end{equation*}%
Further details on the classical theory of these spaces can be found in \cite%
{FJ86}, \cite{FJ90}, \cite{FrJaWe01}, \cite{SchTr87}, \cite{T1} and \cite{T2}%
.

One recognizes immediately that if $\{t_{k}\}=\{2^{sk}\}$, $s\in \mathbb{R}$%
, then 
\begin{equation*}
\dot{B}_{p,q}(\mathbb{R}^{n},\{2^{sk}\})=\dot{B}_{p,q}^{s}(\mathbb{R}%
^{n})\quad \text{and}\quad \dot{F}_{p,q}(\mathbb{R}^{n},\{2^{sk}\})=\dot{F}%
_{p,q}^{s}(\mathbb{R}^{n}).
\end{equation*}%
Moreover, for $\{t_{k}\}=\{2^{sk}w\}$, $s\in \mathbb{R}$ with a weight $w$
we re-obtain the weighted Triebel-Lizorkin spaces; we refer, in particular,
to the papers \cite{Bui82}, \cite{Ry01} and \cite{Sch982} for a
comprehensive treatment of the weighted spaces.

Let $\varphi $, $\psi \in \mathcal{S}(\mathbb{R}^{n})$ satisfying $\mathrm{%
\eqref{Ass1}}$\ through\ $\mathrm{\eqref{Ass3}}$. Recall that the $\varphi $%
-transform $S_{\varphi }$ is defined by setting $(S_{\varphi
}f)_{k,m}=\langle f,\varphi _{k,m}\rangle $ where $\varphi
_{k,m}(x)=2^{kn/2}\varphi (2^{k}x-m)$, $m\in \mathbb{Z}^{n}$ and $k\in 
\mathbb{Z}$. The inverse $\varphi $-transform $T_{\psi }$ is defined by 
\begin{equation*}
T_{\psi }\lambda :=\sum_{k=-\infty }^{\infty }\sum_{m\in \mathbb{Z}%
^{n}}\lambda _{k,m}\psi _{k,m},
\end{equation*}%
where $\lambda =\{\lambda _{k,m}\}_{k\in \mathbb{Z},m\in \mathbb{Z}%
^{n}}\subset \mathbb{C}$, see \cite{FJ90}.

Now we introduce the following sequence spaces.

\begin{definition}
\label{sequence-space}Let $0<p\leq \infty $ and $0<q\leq \infty $. Let $%
\{t_{k}\}$ be a $p$-admissible weight sequence. Then for all complex valued
sequences $\lambda =\{\lambda _{k,m}\}_{k\in \mathbb{Z},m\in \mathbb{Z}%
^{n}}\subset \mathbb{C}$ we define%
\begin{equation*}
\dot{b}_{p,q}(\mathbb{R}^{n},\{t_{k}\}):=\Big\{\lambda :\big\|\lambda |\dot{b%
}_{p,q}(\mathbb{R}^{n},\{t_{k}\})\big\|<\infty \Big\},
\end{equation*}%
where%
\begin{equation*}
\big\|\lambda |\dot{b}_{p,q}(\mathbb{R}^{n},\{t_{k}\})\big\|:=\Big(%
\sum_{k=-\infty }^{\infty }2^{\frac{knq}{2}}\big\|\sum\limits_{m\in \mathbb{Z%
}^{n}}t_{k}\lambda _{k,m}\chi _{k,m}|L_{p}(\mathbb{R}^{n})\big\|^{q}\Big)^{%
\frac{1}{q}}
\end{equation*}%
and 
\begin{equation*}
\dot{f}_{p,q}(\mathbb{R}^{n},\{t_{k}\}):=\Big\{\lambda :\big\|\lambda |\dot{f%
}_{p,q}(\mathbb{R}^{n},\{t_{k}\})\big\|<\infty \Big\}
\end{equation*}%
with $0<p<\infty $, where%
\begin{equation*}
\big\|\lambda |\dot{f}_{p,q}(\mathbb{R}^{n},\{t_{k}\})\big\|:=\Big\|\Big(%
\sum_{k=-\infty }^{\infty }\sum\limits_{m\in \mathbb{Z}^{n}}2^{\frac{knq}{2}%
}t_{k}^{q}|\lambda _{k,m}|^{q}\chi _{k,m}\Big)^{\frac{1}{q}}|L_{p}(\mathbb{R}%
^{n})\Big\|.
\end{equation*}
\end{definition}

Allowing the smoothness $t_{k}$, $k\in \mathbb{Z}$ to vary from point to
point will raise extra difficulties\ to study these function spaces. But by
the following lemma the problem can be reduced to the case of fixed
smoothness.

\begin{proposition}
\label{Equi-norm1}Let $0<p\leq \infty $ and $0<q\leq \infty $. Let $%
\{t_{k}\} $ be a $p$-admissible weight sequence.\newline
$\mathrm{(i)}$ Then%
\begin{equation*}
\big\|\lambda |\dot{b}_{p,q}(\mathbb{R}^{n},\{t_{k}\})\big\|^{\ast }:=\Big(%
\sum_{k=-\infty }^{\infty }2^{\frac{knq}{2}}\Big(\sum\limits_{m\in \mathbb{Z}%
^{n}}|\lambda _{k,m}|^{p}t_{k,m}^{p}\Big)^{\frac{q}{p}}\Big)^{\frac{1}{q}},
\end{equation*}%
\textrm{\ }is an equivalent quasi-norm in $\dot{b}_{p,q}(\mathbb{R}%
^{n},\{t_{k}\})$.\newline
$\mathrm{(ii)}$ Let $0<\theta \leq p<\infty $, $0<q<\infty $ and $0<\kappa
\leq 1$. Assume that $\{t_{k}\}$ satisfying $\mathrm{\eqref{Asum1}}$ with $%
\sigma _{1}=\theta \left( \frac{p}{\theta }\right) ^{\prime }$ and $j=k$.
Then%
\begin{equation*}
\big\|\lambda |\dot{f}_{p,q,\kappa }(\mathbb{R}^{n},\{t_{k}\})\big\|^{\ast
}:=\Big\|\Big(\sum_{k=-\infty }^{\infty }\sum\limits_{m\in \mathbb{Z}%
^{n}}2^{knq(\frac{1}{2}+\frac{1}{\kappa p})}t_{k,m,\kappa }^{q}|\lambda
_{k,m}|^{q}\chi _{k,m}\Big)^{\frac{1}{q}}|L_{p}(\mathbb{R}^{n})\Big\|,
\end{equation*}%
is an equivalent quasi-norm in $\dot{f}_{p,q}(\mathbb{R}^{n},\{t_{k}\})$,
where%
\begin{equation*}
t_{k,m,\kappa }:=\big\|t_{k}|L_{\kappa p}(Q_{k,m})\big\|,\quad k\in \mathbb{Z%
},m\in \mathbb{Z}^{n}.
\end{equation*}
\end{proposition}

\begin{proof}
We prove only (ii) since (i) is obvious. We will proceed in two steps.

\textit{Step 1.} Let us prove that%
\begin{equation}
\big\|\lambda |\dot{f}_{p,q}(\mathbb{R}^{n},\{t_{k}\})\big\|\lesssim \big\|%
\lambda |\dot{f}_{p,q,\kappa }(\mathbb{R}^{n},\{t_{k}\})\big\|^{\ast }.
\label{new-norm}
\end{equation}%
Let $0<\eta <\min (\theta ,q)$. By duality, $\big\|\lambda |\dot{f}_{p,q}(%
\mathbb{R}^{n},\{t_{k}\})\big\|^{\eta }$ is just%
\begin{equation*}
\sup \sum_{k=-\infty }^{\infty }\sum\limits_{m\in \mathbb{Z}^{n}}2^{\frac{%
n\eta }{2}k}|\lambda _{k,m}|^{\eta }\int_{Q_{k,m}}t_{k}^{\eta
}(x)|g_{k}(x)|dx=\sup \sum_{k=-\infty }^{\infty }S_{k},
\end{equation*}%
where the supremum is taking over all sequence of functions $\{g_{k}\}\in
L_{(\frac{p}{\eta })^{\prime }}(\ell _{(\frac{q}{\eta })^{\prime }})$\ with 
\begin{equation}
\big\|\{g_{k}\}|L_{(\frac{p}{\eta })^{\prime }}(\ell _{(\frac{q}{\eta }%
)^{\prime }})\big\|\leq 1.  \label{dual}
\end{equation}%
By H\"{o}lder's inequality,%
\begin{equation*}
1=\Big(\frac{1}{|Q_{k,m}|}\int_{Q_{k,m}}t_{k}^{-\eta h}(x)t_{k}^{\eta h}(x)dx%
\Big)^{\frac{1}{h}}\leq M_{Q_{k,m},\varrho }\mathcal{(}t_{k}^{-\eta })\big(%
M_{Q_{k,m},\kappa p}\mathcal{(}t_{k})\big)^{\eta }
\end{equation*}%
for any $h>0$, with $\frac{1}{h}=\frac{1}{\varrho }+\frac{\eta }{\kappa p}$.
Therefore,%
\begin{equation*}
\int_{Q_{k,m}}t_{k}^{\eta }(x)|g_{k}(x)|dx\lesssim |Q_{k,m}|\big(%
M_{Q_{k,m},\kappa p}\mathcal{(}t_{k})\big)^{\eta }M_{Q_{k,m},\varrho }%
\mathcal{(}t_{k}^{-\eta }\mathcal{M}(t_{k}^{\eta }g_{k})).
\end{equation*}%
We set%
\begin{equation*}
f_{k}(x):=\sum\limits_{m\in \mathbb{Z}^{n}}2^{\frac{kn}{2}}|\lambda
_{k,m}||Q_{k,m}|^{-\frac{1}{\kappa p}}t_{k,m,\kappa }\chi _{k,m}(x),\quad
x\in \mathbb{R}^{n},k\in \mathbb{Z}.
\end{equation*}%
Take $0<\varrho <\min ((\frac{p}{\eta })^{\prime },(\frac{q}{\eta })^{\prime
})$, we find by H\"{o}lder's inequality that the sum $\sum_{k=-\infty
}^{\infty }S_{k}$ can be estimated by 
\begin{align*}
& c\int_{\mathbb{R}^{n}}\sum_{k=-\infty }^{\infty }f_{k}^{\eta }(x)\mathcal{M%
}_{\varrho }(t_{k}^{-\eta }\mathcal{M}(t_{k}^{\eta }g_{k}))(x)dx \\
\lesssim & \Big\|\Big(\sum_{k=-\infty }^{\infty }f_{k}^{q}\Big)^{\frac{\eta 
}{q}}|L_{\frac{p}{\eta }}(\mathbb{R}^{n})\Big\|\Big\|\Big(\sum_{k=-\infty
}^{\infty }\left( \mathcal{M}_{\varrho }(t_{k}^{-\eta }\mathcal{M}%
(t_{k}^{\eta }g_{k}))\right) ^{(\frac{q}{\eta })^{\prime }}\Big)^{\frac{1}{(%
\frac{q}{\eta })^{\prime }}}|L_{(\frac{p}{\eta })^{\prime }}(\mathbb{R}^{n})%
\Big\|.
\end{align*}%
Observe that%
\begin{equation*}
t_{k}^{p}\in A_{\frac{p}{\theta }}\subset A_{\frac{p}{\eta }},\quad k\in 
\mathbb{Z},
\end{equation*}%
which yields that%
\begin{equation*}
t_{k}^{-\eta (\frac{p}{\eta })^{\prime }}\in A_{(\frac{p}{\eta })^{\prime
}},\quad k\in \mathbb{Z},
\end{equation*}%
by Lemma \ref{Ap-Property}/(i)-(ii). By the vector-valued maximal inequality
of Fefferman and Stein $\mathrm{\eqref{Fe-St71}}$, Lemma \ref{key-estimate1}
and $\mathrm{\eqref{dual}}$, we obtain that%
\begin{equation*}
\sum_{k=-\infty }^{\infty }S_{k}\lesssim \Big\|\Big(\sum_{k=-\infty
}^{\infty }f_{k}^{q}\Big)^{\frac{1}{q}}|L_{p}(\mathbb{R}^{n})\Big\|^{\eta },
\end{equation*}%
which yields $\mathrm{\eqref{new-norm}.}$

\textit{Step 2.} We prove the opposite inequality of $\mathrm{%
\eqref{new-norm}}$. We have 
\begin{equation*}
|\lambda _{k,m}|^{\nu \theta }=\frac{1}{|Q_{k,m}|}\int_{Q_{k,m}}|\lambda
_{k,m}|^{\nu \theta }\chi _{k,m}(y)dy,\quad k\in \mathbb{Z},m\in \mathbb{Z}%
^{n}.
\end{equation*}%
Again\ by H\"{o}lder's inequality, since $\frac{\theta }{\sigma _{1}}+\frac{%
\theta }{p}=1$, 
\begin{align*}
|\lambda _{k,m}|\leq & M_{Q_{k,m},\nu p}(\lambda _{k,m}t_{k})M_{Q_{k,m},\nu
\sigma _{1}}(t_{k}^{-1}) \\
\lesssim & \mathcal{M}_{\nu p}\big(\sum\limits_{m\in \mathbb{Z}^{n}}\lambda
_{k,m}t_{k}\chi _{k,m}\big)(x)M_{Q_{k,m},\sigma _{1}}(t_{k}^{-1})
\end{align*}%
for any $x\in Q_{k,m}$ with $0<\nu <\min (1,\frac{q}{p})$. Observing that%
\begin{equation*}
M_{Q_{k,m},\kappa p}(t_{k})\leq M_{Q_{k,m},p}(t_{k}),\quad k\in \mathbb{Z}%
,m\in \mathbb{Z}^{n}.
\end{equation*}%
Hence by $\mathrm{\eqref{Asum1}}$\textrm{, }we find that%
\begin{equation*}
|\lambda _{k,m}|\lesssim 2^{-k\frac{n}{\kappa p}}t_{k,m,\kappa }^{-1}%
\mathcal{M}_{\nu p}\big(\sum\limits_{m\in \mathbb{Z}^{n}}t_{k}\lambda
_{k,m}\chi _{k,m}\big)(x)
\end{equation*}%
for any $x\in Q_{k,m}$. Therefore,%
\begin{align*}
\big\|\lambda |\dot{f}_{p,q,\kappa }(\mathbb{R}^{n},\{t_{k}\})\big\|^{\ast
}\lesssim & \Big\|\Big(\sum_{k=-\infty }^{\infty }\big(\mathcal{M}_{\nu
p}(\sum\limits_{m\in \mathbb{Z}^{n}}2^{\frac{kn}{2}}t_{k}\lambda _{k,m}\chi
_{k,m})\big)^{q}\Big)^{\frac{1}{q}}|L_{p}(\mathbb{R}^{n})\Big\| \\
\lesssim & \big\|\lambda |\dot{f}_{p,q}(\mathbb{R}^{n},\{t_{k}\})\big\|,
\end{align*}%
where we used the vector-valued maximal inequality of Fefferman and Stein $%
\mathrm{\eqref{Fe-St71}}$. The proof is complete.
\end{proof}

For simplicity, in what follows, we use $\dot{a}_{p,q}(\mathbb{R}%
^{n},\{t_{k}\})$ to denote either $\dot{b}_{p,q}(\mathbb{R}^{n},\{t_{k}\})$
or $\dot{f}_{p,q}(\mathbb{R}^{n},\{t_{k}\})$. Now we have the following
result which is called the $\varphi $-transform characterization in the
sense of Frazier and Jawerth. It will play an important role in the rest of
the paper. The proof is given in {\cite{D20} and \cite{D20.1}.}

\begin{theorem}
\label{phi-tran}Let $\alpha =(\alpha _{1},\alpha _{2})\in \mathbb{R}%
^{2},0<\theta \leq p<\infty $ and$\ 0<q<\infty $. Let $\{t_{k}\}\in \dot{X}%
_{\alpha ,\sigma ,p}$ be a $p$-admissible weight sequence with $\sigma
=(\sigma _{1}=\theta \left( \frac{p}{\theta }\right) ^{\prime },\sigma
_{2}\geq p)$.\ Let $\varphi $, $\psi \in \mathcal{S}(\mathbb{R}^{n})$
satisfying $\mathrm{\eqref{Ass1}}$\ through\ $\mathrm{\eqref{Ass3}}$. The
operators 
\begin{equation*}
S_{\varphi }:\dot{A}_{p,q}(\mathbb{R}^{n},\{t_{k}\})\rightarrow \dot{a}%
_{p,q}(\mathbb{R}^{n},\{t_{k}\})
\end{equation*}%
and 
\begin{equation*}
T_{\psi }:\dot{a}_{p,q}(\mathbb{R}^{n},\{t_{k}\})\rightarrow \dot{A}_{p,q}(%
\mathbb{R}^{n},\{t_{k}\})
\end{equation*}%
are bounded. Furthermore, $T_{\psi }\circ S_{\varphi }$ is the identity on $%
\dot{A}_{p,q}(\mathbb{R}^{n},\{t_{k}\})$.
\end{theorem}

\begin{remark}
\label{phi-tran1}This theorem can then be exploited to obtain a variety of
results for the $\dot{F}_{p,q}(\mathbb{R}^{n},\{t_{k}\})$ spaces, where
arguments can be equivalently transferred to the sequence space, which is
often more convenient to handle. More precisely, under the same hypothesis
of the last theorem, 
\begin{equation*}
\big\|\{\langle f,\varphi _{k,m}\rangle \}_{k\in \mathbb{Z},m\in \mathbb{Z}%
^{n}}|\dot{a}_{p,q}(\mathbb{R}^{n},\{t_{k}\})\big\|\approx \big\|\{f|\dot{A}%
_{p,q}(\mathbb{R}^{n},\{t_{k}\})\big\|.
\end{equation*}
\end{remark}

\begin{corollary}
\label{Indpendent}Let $\alpha =(\alpha _{1},\alpha _{2})\in \mathbb{R}%
^{2},0<\theta \leq p<\infty $ and $0<q<\infty $. Let $\{t_{k}\}\in \dot{X}%
_{\alpha ,\sigma ,p}$ be a $p$-admissible weight sequence with $\sigma
=(\sigma _{1}=\theta \left( \frac{p}{\theta }\right) ^{\prime },\sigma
_{2}\geq p)$. The definition of the spaces $\dot{A}_{p,q}(\mathbb{R}%
^{n},\{t_{k}\})$ is independent of the choices of $\varphi \in \mathcal{S}(%
\mathbb{R}^{n})$ satisfying $\mathrm{\eqref{Ass1}}$\ and\ $\mathrm{%
\eqref{Ass2}}$.
\end{corollary}

\begin{theorem}
Let $\alpha =(\alpha _{1},\alpha _{2})\in \mathbb{R}^{2},0<\theta \leq
p<\infty $ and $0<q<\infty $. Let $\{t_{k}\}\in \dot{X}_{\alpha ,\sigma ,p}$
be a $p$-admissible weight sequence with $\sigma =(\sigma _{1}=\theta \left( 
\frac{p}{\theta }\right) ^{\prime },\sigma _{2}\geq p)$. $\dot{A}_{p,q}(%
\mathbb{R}^{n},\{t_{k}\})$ are quasi-Banach spaces. They are Banach spaces
if $1\leq p<\infty $ and $1\leq q<\infty $.
\end{theorem}

\begin{theorem}
\label{embeddings-S-inf}Let $0<\theta \leq p<\infty $ and $0<q<\infty $.
Let\ $\{t_{k}\}\in \dot{X}_{\alpha ,\sigma ,p}$ be a $p$-admissible weight
sequence with $\sigma =(\sigma _{1}=\theta \left( \frac{p}{\theta }\right)
^{\prime },\sigma _{2}\geq p)$ and $\alpha =(\alpha _{1},\alpha _{2})\in 
\mathbb{R}^{2}$.\ \newline
$\mathrm{(i)}$ We have the embedding%
\begin{equation*}
\mathcal{S}_{\infty }(\mathbb{R}^{n})\hookrightarrow \dot{A}_{p,q}(\mathbb{R}%
^{n},\{t_{k}\})\hookrightarrow \mathcal{S}_{\infty }^{\prime }(\mathbb{R}%
^{n}).
\end{equation*}%
In addition $\mathcal{S}_{\infty }(\mathbb{R}^{n})$ is dense in $\dot{A}%
_{p,q}(\mathbb{R}^{n},\{t_{k}\}\mathrm{.}$
\end{theorem}

The proof is given in {\cite{D20} and \cite{D20.1}. }Now we recall the
following spaces, see \cite{D20.2}.

\begin{definition}
\label{F-inf-def}Let $0<q<\infty $.\ Let $\{t_{k}\}$ be a $q$-admissible
weight sequence and $\varphi \in \mathcal{S}(\mathbb{R}^{n})$\ satisfy $%
\mathrm{\eqref{Ass1}}$ and $\mathrm{\eqref{Ass2}}$ and we put $\varphi
_{k}=2^{kn}\varphi (2^{k}\cdot )$. The Triebel-Lizorkin\ space $\dot{F}%
_{\infty ,q}(\mathbb{R}^{n},\{t_{k}\})$\ is the collection of all $f\in 
\mathcal{S}_{\infty }^{\prime }(\mathbb{R}^{n})$\ such that 
\begin{equation*}
{\big\|}f|\dot{F}_{\infty ,q}(\mathbb{R}^{n},\{t_{k}\}){\big\|}:=\sup_{P\in 
\mathcal{Q}}\Big(\frac{1}{|P|}\int_{P}\sum\limits_{k=-\log _{2}l(P)}^{\infty
}t_{k}^{q}(x)|\varphi _{k}\ast f(x)|^{q}dx\Big)^{\frac{1}{q}}<\infty .
\end{equation*}%
We define $\dot{f}_{\infty ,q}(\mathbb{R}^{n},\{t_{k}\})$, the sequence
space corresponding to\ $\dot{F}_{\infty ,q}(\mathbb{R}^{n},\{t_{k}\})$ as
follows.
\end{definition}

\begin{definition}
Let $0<q<\infty $ and $\{t_{k}\}$ be a $q$-admissible sequence. Then for all
complex valued sequences $\lambda =\{\lambda _{k,m}\}_{k\in \mathbb{Z},m\in 
\mathbb{Z}^{n}}\subset \mathbb{C}$ we define 
\begin{equation*}
\dot{f}_{\infty ,q}(\mathbb{R}^{n},\{t_{k}\}):=\Big\{\lambda :{\big\|}%
\lambda |\dot{f}_{\infty ,q}(\mathbb{R}^{n},\{t_{k}\}){\big\|}<\infty \Big\},
\end{equation*}%
where%
\begin{equation}
{\big\|}\lambda |\dot{f}_{\infty ,q}(\mathbb{R}^{n},\{t_{k}\}){\big\|:}%
=\sup_{P\in \mathcal{Q}}\Big(\frac{1}{|P|}\int_{P}\sum\limits_{k=-\log
_{2}l(P)}^{\infty }\sum\limits_{m\in \mathbb{Z}^{n}}2^{\frac{knq}{2}%
}t_{k}^{q}(x)|\lambda _{k,m}|^{q}\chi _{k,m}(x)dx\Big)^{\frac{1}{q}}.
\label{norm}
\end{equation}
\end{definition}

The quasi-norm $\mathrm{\eqref{norm}}$ can be rewritten as follows:

\begin{proposition}
Let $0<\theta \leq q<\infty $. Let $\{t_{k}\}$ be a $q$-admissible sequence.
Then 
\begin{equation*}
\big\|\lambda |\dot{f}_{\infty ,q}(\mathbb{R}^{n},\{t_{k}\})\big\|%
=\sup_{P\in \mathcal{Q}}\Big(\frac{1}{|P|}\int_{P}\sum\limits_{k=-\log
_{2}l(P)}^{\infty }\sum\limits_{m\in \mathbb{Z}^{n}}2^{knq(\frac{1}{2}+\frac{%
1}{q})}t_{k,m,q}^{q}|\lambda _{k,m}|^{q}\chi _{k,m}(x)dx\Big)^{\frac{1}{q}},
\end{equation*}%
where%
\begin{equation*}
t_{k,m,q}:=\big\|t_{k}|L_{q}(Q_{k,m})\big\|,\quad k\in \mathbb{Z},m\in 
\mathbb{Z}^{n}.
\end{equation*}
\end{proposition}

\begin{proof}
Let $Q_{k,m}\subset P\in \mathcal{Q},k\in \mathbb{Z},m\in \mathbb{Z}^{n}$.
The claim is a simple consequence of the fact that%
\begin{align*}
& \int_{P}\sum\limits_{k=-\log _{2}l(P)}^{\infty }\sum\limits_{m\in \mathbb{Z%
}^{n}}2^{\frac{knq}{2}}t_{k}^{q}(x)|\lambda _{k,m}|^{q}\chi _{k,m}(x)dx \\
=& \sum\limits_{k=-\log _{2}l(P)}^{\infty }\sum\limits_{m\in \mathbb{Z}%
^{n}}2^{\frac{knq}{2}}|\lambda _{k,m}|^{q}\int_{P}t_{k}^{q}(x)\chi
_{k,m}(x)dx \\
=& \sum\limits_{k=-\log _{2}l(P)}^{\infty }\sum\limits_{m\in \mathbb{Z}%
^{n}}2^{\frac{knq}{2}}|\lambda _{k,m}|^{q}t_{k,m,q}^{q} \\
=& |Q_{k,m}|^{-1}\int_{Q_{k,m}}\sum\limits_{k=-\log _{2}l(P)}^{\infty
}\sum\limits_{m\in \mathbb{Z}^{n}}2^{\frac{knq}{2}}|\lambda
_{k,m}|^{q}t_{k,m,q}^{q}dx \\
=& \int_{P}\sum\limits_{k=-\log _{2}l(P)}^{\infty }\sum\limits_{m\in \mathbb{%
Z}^{n}}2^{knq(\frac{1}{2}+\frac{1}{q})}|\lambda _{k,m}|^{q}t_{k,m,q}^{q}\chi
_{k,m}(x)dx.
\end{align*}
\end{proof}

We have, see \cite{FJ90}, 
\begin{equation*}
\dot{F}_{\infty ,2}(\mathbb{R}^{n},\{1\})=BMO(\mathbb{R}^{n}).
\end{equation*}%
Notice that Theorem \ref{phi-tran} is true for the spaces $\dot{F}_{\infty
,q}(\mathbb{R}^{n},\{t_{k}\})$ and $\dot{f}_{\infty ,q}(\mathbb{R}%
^{n},\{t_{k}\})$, see \cite{D20.2}. More precisely:

\begin{theorem}
\label{phi-tran2}Let $0<\theta \leq q<\infty $. Let $\{t_{k}\}\in \dot{X}%
_{\alpha ,\sigma ,q}$ be a $q$-admissible weight sequence with $\sigma
=(\sigma _{1}=\theta \left( \frac{q}{\theta }\right) ^{\prime },\sigma
_{2}\geq q)$. Let $\varphi $, $\psi \in \mathcal{S}(\mathbb{R}^{n})$
satisfying $\mathrm{\eqref{Ass1}}$\ through\ $\mathrm{\eqref{Ass3}}$. The
operators 
\begin{equation*}
S_{\varphi }:\dot{F}_{\infty ,q}(\mathbb{R}^{n},\{t_{k}\})\rightarrow \dot{f}%
_{\infty ,q}(\mathbb{R}^{n},\{t_{k}\})
\end{equation*}%
and 
\begin{equation*}
T_{\psi }:\dot{f}_{\infty ,q}(\mathbb{R}^{n},\{t_{k}\})\rightarrow \dot{F}%
_{\infty ,q}(\mathbb{R}^{n},\{t_{k}\})
\end{equation*}%
are bounded. Furthermore, $T_{\psi }\circ S_{\varphi }$ is the identity on $%
\dot{F}_{\infty ,q}(\mathbb{R}^{n},\{t_{k}\})$.
\end{theorem}

\begin{corollary}
Let $0<\theta \leq q<\infty $. Let $\{t_{k}\}\in \dot{X}_{\alpha ,\sigma ,q}$
be a $q$-admissible weight sequence with $\sigma =(\sigma _{1}=\theta \left( 
\frac{q}{\theta }\right) ^{\prime },\sigma _{2}\geq q)$. The definition of
the spaces $\dot{F}_{\infty ,q}(\mathbb{R}^{n},\{t_{k}\})$ is independent of
the choices of $\varphi \in \mathcal{S}(\mathbb{R}^{n})$ satisfying $\mathrm{%
\eqref{Ass1}}$\ through\ $\mathrm{\eqref{Ass2}}$.
\end{corollary}

As in \cite{FJ90} we obtain the following statement.

\begin{proposition}
\label{norm-equiv11}Let $0<\theta \leq q<\infty $. Let $\{t_{k}\}$ be a $q$%
-admissible weight sequence satisfying $\mathrm{\eqref{Asum1}}$ with $\sigma
_{1}=\theta \left( \frac{q}{\theta }\right) ^{\prime }$, $p=q$ and $j=k$%
\textit{. Then }$\lambda =\{\lambda _{k,m}\}_{k\in \mathbb{Z},m\in \mathbb{Z}%
^{n}}\in \dot{f}_{\infty ,q}(\mathbb{R}^{n},\{t_{k}\})$ \textit{if and only
if for each \ dyadic cube }$Q_{k,m}$ there is a subset $E_{Q_{k,m}}\subset
Q_{k,m}$ with $|E_{Q_{k,m}}|>|Q_{k,m}|/2$ (or any other, fixed, number $%
0<\varepsilon <1$) such that%
\begin{equation*}
\Big\|\Big(\sum_{k=-\infty }^{\infty }\sum\limits_{m\in \mathbb{Z}%
^{n}}2^{knq(\frac{1}{2}+\frac{1}{q})}t_{k,m,q}^{q}|\lambda _{k,m}|^{q}\chi
_{E_{Q_{k,m}}}\Big)^{1/q}|L_{\infty }(\mathbb{R}^{n})\Big\|<\infty .
\end{equation*}%
Moreover, the infimum of this expression over all such collections $%
\{E_{Q_{k,m}}\}_{k,m}$ is equivalent to ${\big\|}\lambda |\dot{f}_{\infty
,q}(\mathbb{R}^{n},\{t_{k}\}){\big\|}$.
\end{proposition}

\section{Complex interpolation}

In this section we study complex interpolation of the above function spaces
using Calder\'{o}n product\ method. We follow the approach of Frazier and
Jawerth \cite{FJ90}, see also \cite{SSV13}. We start by defining the Calder%
\'{o}n product of two quasi-Banach lattices. Let $(\mathcal{A},S,%
\mu
)$ be a $\sigma $-finite measure space and let $\mathfrak{M}$ be the class
of all complex-valued, $\mu $-measurable functions on $\mathcal{A}$. Then a
quasi-Banach space $X\subset \mathfrak{M}$ is called a quasi-Banach lattice
of functions if for every $f\in X$ and $g\in \mathfrak{M}$ with $|g(x)|\leq
|f(x)|$ for $\mu $-a.e. $x\in X$ one has $g\in X$ and $\left\Vert
g\right\Vert _{X}\leq \left\Vert f\right\Vert _{X}$.

\begin{definition}
Let $(\mathcal{A},S,%
\mu
)$ be a $\sigma $-finite measure space and let $\mathfrak{M}$ be the class
of all complex-valued, $\mu $-measurable functions on $\mathcal{A}$. Suppose
that $X_{0}$ and $X_{1}$ are quasi-Banach lattices on $\mathfrak{M}$. Given $%
0<\theta <1$, define the Calder\'{o}n product $X_{0}^{1-\theta }\cdot
X_{1}^{\theta }$ as the collection of all functions $f\in \mathfrak{M}$ such
that%
\begin{equation*}
\left\Vert f\right\Vert _{X_{0}^{1-\theta }\cdot X_{1}^{\theta }}:=\inf %
\big\{\left\Vert g\right\Vert _{X_{0}}^{1-\theta }\left\Vert h\right\Vert
_{X_{1}}^{\theta }:|f|\leq |g|^{1-\theta }|h|^{\theta },\mu \text{-}a.e.,\
g\in X_{0},\ h\in X_{1}\big\}<\infty .
\end{equation*}
\end{definition}

\begin{remark}
Calder\'{o}n products have been introduced by Calder\'{o}n \cite{Ca64} (in a
little bit different form which coincides with the above one). Further
properties we refer to, Frazier and Jawerth \cite{FJ90} and Yang, Yuan and
Zhuo \cite{YYZ13}.
\end{remark}

We need a few useful properties, see \cite{YYZ13}.

\begin{lemma}
Let $(\mathcal{A},S,%
\mu
)$ be a $\sigma $-finite measure space and let $\mathfrak{M}$ be the class
of all complex-valued, $\mu $-measurable functions on $\mathcal{A}$. Suppose
that $X_{0}$ and $X_{1}$ are quasi-Banach lattices on $\mathfrak{M}$. Let $%
0<\theta <1$.

$\mathrm{(i)}$ Then the Calder\'{o}n product $X_{0}^{1-\theta }\cdot
X_{1}^{\theta }$ is a quasi-Banach space.

$\mathrm{(ii)}$ Define the Calder\'{o}n product $\widetilde{X_{0}^{1-\theta
}\cdot X_{1}^{\theta }}$ as the collection of all functions $f\in \mathfrak{M%
}$ such that there exist a positive real number $M$ and $g\in X_{0}$ and $%
h\in X_{1}$ satisfying%
\begin{equation*}
|f|\leq M|g|^{1-\theta }|h|^{\theta },\quad \left\Vert g\right\Vert
_{X_{0}}\leq 1,\quad \left\Vert h\right\Vert _{X_{1}}\leq 1.
\end{equation*}%
We put%
\begin{equation*}
\left\Vert f\right\Vert _{\widetilde{X_{0}^{1-\theta }\cdot X_{1}^{\theta }}%
}:=\inf \big\{M>0:|f|\leq M|g|^{1-\theta }|h|^{\theta },\ \left\Vert
g\right\Vert _{X_{0}}\leq 1,\ \left\Vert h\right\Vert _{X_{1}}\leq 1\big\}.
\end{equation*}%
Then $\widetilde{X_{0}^{1-\theta }\cdot X_{1}^{\theta }}=X_{0}^{1-\theta
}\cdot X_{1}^{\theta }$ follows with equality of quasi-norms.
\end{lemma}

In the sequel we will need the following lemma:

\begin{lemma}
\label{Holder}Let $0<\theta <1,0<1-\theta <q_{0}<\infty ,0<\varrho
<q_{1}<\infty $ and $Q$ be a cube. Let $\{t_{k}\}\subset L_{q_{0}}^{\mathrm{%
loc}}$ and $\{w_{k}\}\subset L_{q_{1}}^{\mathrm{loc}}$ be two\ weight
sequences satisfying $\mathrm{\eqref{Asum1}}$ with $(\sigma
_{1},p,j)=((1-\theta )(\frac{q_{0}}{1-\theta })^{\prime },p=q_{0},k)$\ and $%
(\sigma _{1},p,j)=(\theta (\frac{q_{1}}{\theta })^{\prime },p=q_{1},k)$,
respectively. We put%
\begin{equation*}
\frac{1}{q}:=\frac{1-\theta }{q_{0}}+\frac{\theta }{q_{1}}\text{\quad
and\quad }\omega _{k}:=t_{k}^{1-\theta }w_{k}^{\theta },\quad k\in \mathbb{Z}%
.
\end{equation*}%
Then%
\begin{equation*}
\Big(\int_{E}\omega _{k}^{q}(x)dx\Big)^{\frac{1}{q}}\approx \Big(%
\int_{E}t_{k}^{q_{0}}(x)dx\Big)^{\frac{1-\theta }{q_{0}}}\Big(%
\int_{E}w_{k}^{q_{1}}(x)dx\Big)^{\frac{\theta }{q_{1}}}
\end{equation*}%
for any $E\subset Q$ such that $|E|\geq \varepsilon |Q|$, with $\varepsilon
>0$.
\end{lemma}

\begin{proof}
First this lemma with $E=Q$ is given in \cite{SSV13}. Let $0<\delta ,\mu <1$
and assume that $q_{0}<q$. Since, $\frac{1}{1-\theta }=\frac{1}{q_{0}}+\frac{%
1}{\sigma _{1}}$, H\"{o}lder's inequality implies the following estimate%
\begin{align*}
1=& \Big(\frac{1}{|E|}\int_{E}t_{k}^{\delta (1-\theta )}(x)t_{k}^{-\delta
(1-\theta )}(x)dx\Big)^{\frac{1}{\delta }} \\
\leq & \frac{1}{\varepsilon ^{\frac{1}{\delta }}}\Big(\frac{1}{|Q|}%
\int_{Q}t_{k}^{\delta q_{0}}(x)\chi _{E}(x)dx\Big)^{\frac{1-\theta }{\delta
q_{0}}}\Big(\frac{1}{|Q|}\int_{Q}t_{k}^{-\delta \sigma _{1}}(x)\chi _{E}(x)dx%
\Big)^{\frac{1-\theta }{\delta \sigma _{1}}}.
\end{align*}%
Since $\{t_{k}\}$ is a $q_{0}$-admissible sequence satisfying $\mathrm{%
\eqref{Asum1}}$ with $\sigma _{1}=(1-\theta )(\frac{q_{0}}{1-\theta }%
)^{\prime }$ and $p=q_{0}$, we estimate the above expression by%
\begin{equation*}
\frac{c}{\varepsilon ^{\frac{1}{\delta }}}\Big(\frac{1}{|Q|}%
\int_{Q}t_{k}^{\delta q_{0}}(x)\chi _{E}(x)dx\Big)^{\frac{1-\theta }{\delta
q_{0}}}\Big(\frac{1}{|Q|}\int_{Q}t_{k}^{q_{0}}(x)\chi _{E}(x)dx\Big)^{-\frac{%
1-\theta }{q_{0}}},
\end{equation*}%
where the positive constant $c$ not depending on $k$ and $Q$.\ Similarly for 
$w_{k},k\in \mathbb{Z}$. Therefore 
\begin{align}
& |Q|^{-\frac{1}{q}}\Big(\int_{Q}t_{k}^{q_{0}}(x)\chi _{E}(x)dx\Big)^{\frac{%
1-\theta }{q_{0}}}\Big(\int_{Q}w_{k}^{q_{1}}(x)\chi _{E}(x)dx\Big)^{\frac{%
\theta }{q_{1}}}  \notag \\
\lesssim & \Big(\frac{1}{|Q|}\int_{Q}t_{k}^{\delta q_{0}}(x)\chi _{E}(x)dx%
\Big)^{\frac{1-\theta }{\delta q_{0}}}\Big(\frac{1}{|Q|}\int_{Q}w_{k}^{\mu
q_{1}}(x)\chi _{E}(x)dx\Big)^{\frac{\theta }{\mu q_{1}}}.
\label{est-t-and-w}
\end{align}%
Take $0<\mu <\frac{\theta (1-\theta )}{q_{1}}$. Using the fact that%
\begin{equation*}
\Big(\frac{1}{|Q|}\int_{Q}w_{k}^{\mu q_{1}}(x)\chi _{E}(x)dx\Big)^{\frac{%
\theta }{\mu q_{1}}}\leq \Big(\frac{1}{|Q|}\int_{Q}w_{k}^{\theta (1-\theta
)}(x)\chi _{E}(x)dx\Big)^{\frac{1}{1-\theta }}\lesssim \mathcal{M}_{1-\theta
}(w_{k}^{\theta }\chi _{E})(y)
\end{equation*}%
for any\ $y\in Q$ and taking $\delta =\frac{(1-\theta )q}{q_{0}}$, we find
that $\mathrm{\eqref{est-t-and-w}}$ is bounded by 
\begin{equation*}
c\big(\frac{1}{|Q|}\big)^{\frac{1-\theta }{\delta q_{0}}}\big\|%
t_{k}^{1-\theta }\mathcal{M}_{1-\theta }(w_{k}^{\theta }\chi _{E})|L_{\frac{%
\delta q_{0}}{1-\theta }}(\mathbb{R}^{n})\big\|=c\big(\frac{1}{|Q|}\big)^{%
\frac{1}{q}}\big\|t_{k}^{1-\theta }\mathcal{M}_{1-\theta }(w_{k}^{\theta
}\chi _{E})|L_{q}(\mathbb{R}^{n})\big\|.
\end{equation*}%
Observing that%
\begin{equation*}
t_{k}^{q_{0}}\in A_{\frac{q_{0}}{1-\theta }}\subset A_{\frac{q}{1-\theta }%
},\quad k\in \mathbb{Z}.
\end{equation*}%
Therefore, since $0<\frac{(1-\theta )q}{q_{0}}<1$, we obtain 
\begin{equation*}
t_{k}^{(1-\theta )q}\in A_{\frac{q}{1-\theta }},\quad k\in \mathbb{Z}.
\end{equation*}%
From Lemma \ref{key-estimate1} combined with Remark \ref{r-estimates}/(iii),%
\begin{align*}
\big\|t_{k}^{1-\theta }\mathcal{M}_{1-\theta }(w_{k}^{\theta }\chi
_{E})|L_{q}(\mathbb{R}^{n})\big\|\lesssim & \big\|t_{k}^{1-\theta
}w_{k}^{\theta }\chi _{E}|L_{q}(\mathbb{R}^{n})\big\| \\
=& c\Big(\int_{E}\omega _{k}^{q}(x)dx\Big)^{\frac{1}{q}}.
\end{align*}%
Then%
\begin{equation*}
\Big(\int_{Q}t_{k}^{q_{0}}(x)\chi _{E}(x)dx\Big)^{\frac{1-\theta }{q_{0}}}%
\Big(\int_{Q}w_{k}^{q_{1}}(x)\chi _{E}(x)dx\Big)^{\frac{\theta }{q_{1}}%
}\lesssim \Big(\int_{E}\omega _{k}^{q}(x)dx\Big)^{\frac{1}{q}}.
\end{equation*}%
The rest inequality follows by H\"{o}lder's inequality.
\end{proof}

Now we turn to the investigation of the Calder\'{o}n products of the
sequence spaces $\dot{a}_{p,q}(\mathbb{R}^{n},\{t_{k}\})$.

\begin{theorem}
\label{calderon-prod1}Let $0<\theta <1,1\leq p_{0},p_{1}<\infty $ and $1\leq
q_{0},q_{1}<\infty $. Let $\{t_{k}\}\subset L_{p_{0}}^{\mathrm{loc}}$ and $%
\{w_{k}\}\subset L_{p_{1}}^{\mathrm{loc}}$ be two weight sequences
satisfying $\mathrm{\eqref{Asum1}}$ with $(\sigma _{1},p,j)=((1-\theta )(%
\frac{p_{0}}{1-\theta })^{\prime },p=p_{0},k)$\ and $(\sigma
_{1},p,j)=(\theta (\frac{p_{1}}{\theta })^{\prime },p=p_{1},k)$,
respectively. We put%
\begin{equation}
\frac{1}{p}:=\frac{1-\theta }{p_{0}}+\frac{\theta }{p_{1}},\quad \frac{1}{q}%
:=\frac{1-\theta }{q_{0}}+\frac{\theta }{q_{1}},\quad \omega
_{k}:=t_{k}^{1-\theta }w_{k}^{\theta },\quad k\in \mathbb{Z}.
\label{Th-Cond}
\end{equation}%
Then%
\begin{equation*}
\big(\dot{a}_{p_{0},q_{0}}(\mathbb{R}^{n},\{t_{k}\})\big)^{1-\theta }\big(%
\dot{a}_{p_{1},q_{1}}(\mathbb{R}^{n},\{w_{k}\})\big)^{\theta }=\dot{a}_{p,q}(%
\mathbb{R}^{n},\{\omega _{k}\})
\end{equation*}%
holds in the sense of equivalent norms.
\end{theorem}

\begin{proof}
Obviously we can assume that $1<p_{0},p_{1},q_{0},q_{1}<\infty $. Put%
\begin{equation*}
t_{k,m}:=\big\|t_{k}|L_{p_{0}}(Q_{k,m})\big\|,\quad w_{k,m}:=\big\|%
w_{k}|L_{p_{1}}(Q_{k,m})\big\|
\end{equation*}%
and%
\begin{equation*}
\omega _{k,m}:=\big\|\omega _{k}|L_{p}(Q_{k,m})\big\|,\quad k\in \mathbb{Z}%
,m\in \mathbb{Z}^{n}.
\end{equation*}%
\textit{Step 1.} We deal with the case of $\dot{f}$-spaces. To prove we
additionally do it into the two Substeps 1.1 and 1.2.

\noindent \textit{Substep 1.1. }We shall prove%
\begin{equation*}
\big(\dot{f}_{p_{0},q_{0}}(\mathbb{R}^{n},\{t_{k}\})\big)^{1-\theta }\big(%
\dot{f}_{p_{1},q_{1}}(\mathbb{R}^{n},\{w_{k}\})\big)^{\theta
}\hookrightarrow \dot{f}_{p,q}(\mathbb{R}^{n},\{\omega _{k}\}).
\end{equation*}%
We suppose, that sequences $\lambda :=(\lambda _{k,m})_{k,m}$, $\lambda
^{i}:=(\lambda _{k,m}^{i})_{j,m},i=0,1$, are given and that%
\begin{equation*}
|\lambda _{k,m}|\leq |\lambda _{k,m}^{0}|^{1-\theta }|\lambda
_{k,m}^{1}|^{\theta }
\end{equation*}%
holds for all $k\in \mathbb{Z}$ and $m\in \mathbb{Z}^{n}$. Let 
\begin{equation*}
g:=\Big(\sum_{k=-\infty }^{\infty }\sum\limits_{m\in \mathbb{Z}^{n}}2^{k%
\frac{n}{2}q}\omega _{k}^{q}|\lambda _{k,m}|^{q}\chi _{k,m}\Big)^{\frac{1}{q}%
}.
\end{equation*}%
Since, 
\begin{equation*}
2^{k\frac{n}{2}q}\omega _{k}^{q}|\lambda _{k,m}|^{q}\chi _{k,m}\leq \big(2^{k%
\frac{n}{2}}t_{k}|\lambda _{k,m}^{0}|\chi _{k,m}\big)^{q(1-\theta )}\big(2^{k%
\frac{n}{2}}w_{k}|\lambda _{k,m}^{1}|\chi _{k,m}\big)^{q\theta },
\end{equation*}%
H\"{o}lder's inequality implies that $g$ can be estimated by%
\begin{align*}
& \Big(\sum_{k=-\infty }^{\infty }\sum\limits_{m\in \mathbb{Z}^{n}}\big(2^{k%
\frac{n}{2}}t_{k}|\lambda _{k,m}^{0}|\chi _{j,m}\big)^{q_{0}}\Big)^{\frac{%
1-\theta }{q_{0}}} \\
& \times \Big(\sum_{k=-\infty }^{\infty }\sum\limits_{m\in \mathbb{Z}^{n}}%
\big(2^{k\frac{n}{2}}w_{k}|\lambda _{k,m}^{1}|\chi _{j,m}\big)^{q_{1}}\Big)^{%
\frac{\theta }{q_{1}}}.
\end{align*}%
Applying H\"{o}lder's inequality again with conjugate indices $\frac{p_{0}}{%
1-\theta }$ and $\frac{p_{1}}{\theta }$, we obtain%
\begin{equation*}
\big\|\lambda |\dot{f}_{p,q}(\mathbb{R}^{n},\{\omega _{k}\})\big\|\leq \big\|%
\lambda ^{0}|\dot{f}_{p_{0},q_{0}}(\mathbb{R}^{n},\{t_{k}\})\big\|^{1-\theta
}\big\|\lambda ^{1}|\dot{f}_{p_{1},q_{1}}(\mathbb{R}^{n},\{w_{k}\})\big\|%
^{\theta }.
\end{equation*}%
\textit{Substep 1.2. }Now we turn to the proof of%
\begin{equation*}
\dot{f}_{p,q}(\mathbb{R}^{n},\{\omega _{k}\})\hookrightarrow \big(\dot{f}%
_{p_{0},q_{0}}(\mathbb{R}^{n},\{t_{k}\})\big)^{1-\theta }\big(\dot{f}%
_{p_{1},q_{1}}(\mathbb{R}^{n},\{w_{k}\})\big)^{\theta }.
\end{equation*}%
\textit{Substep 1.2.1.} We consider the case $\gamma =\frac{p}{p_{0}}-\frac{q%
}{q_{0}}>0$. Let the sequence $\lambda \in \dot{f}_{p,q}(\mathbb{R}%
^{n},\{\omega _{k}\})$ be given. We have to find sequences $\lambda ^{0}$
and $\lambda ^{1}$ such that 
\begin{equation*}
|\lambda _{k,m}|\leq M|\lambda _{k,m}^{0}|^{1-\theta }|\lambda
_{k,m}^{1}|^{\theta }
\end{equation*}%
for every $k\in \mathbb{Z}$, $m\in \mathbb{Z}^{n}$ and 
\begin{equation}
\big\|\lambda ^{0}|\dot{f}_{p_{0},q_{0}}(\mathbb{R}^{n},\{t_{k}\})\big\|%
^{1-\theta }\big\|\lambda ^{1}|\dot{f}_{p_{1},q_{1}}(\mathbb{R}%
^{n},\{w_{k}\})\big\|^{\theta }\leq c\big\|\lambda |\dot{f}_{p,q}(\mathbb{R}%
^{n},\{\omega _{k}\})\big\|,  \label{Est-Int}
\end{equation}%
with some constant $c$ independent of $\lambda $.

\textit{Preparation.} We set 
\begin{equation*}
g:=\Big(\sum_{k=-\infty }^{\infty }\sum\limits_{m\in \mathbb{Z}^{n}}2^{kn(%
\frac{1}{p}+\frac{1}{2})q}\omega _{k,m}^{q}|\lambda _{k,m}|^{q}\chi
_{Q_{k,m}}\Big)^{\frac{1}{q}}.
\end{equation*}%
We follow ideas of the proof of Theorem 8.2 in Frazier and Jawerth \cite%
{FJ90}, see also Sickel, Skrzypczak\ and\ Vyb\'{\i}ral \cite{SSV13}. Set 
\begin{equation*}
A_{\ell }:=\{x\in \mathbb{R}^{n}:g(x)>2^{\ell }\},
\end{equation*}%
with $\ell \in \mathbb{Z}$. Obviously $A_{\ell +1}\subset A_{\ell }$, with $%
\ell \in \mathbb{Z}$. Now we introduce a (partial) decomposition of $\mathbb{%
Z}\times \mathbb{Z}^{n}$ by taking%
\begin{equation*}
C_{\ell }:=\{(k,m):|Q_{k,m}\cap A_{\ell }|>\frac{|Q_{k,m}|}{2}\text{ \ and \ 
}|Q_{k,m}\cap A_{\ell +1}|\leq \frac{|Q_{k,m}|}{2}\},\quad \ell \in \mathbb{Z%
}.
\end{equation*}%
The sets $C_{\ell }$ are pairwise disjoint, i.e., $C_{\ell }\cap
C_{v}=\emptyset $ if $\ell \neq v$. Let us prove that $\lambda _{k,m}=0$
holds for all tuples $(k,m)\notin \cup _{\ell \in \mathbb{Z}}C_{\ell }$. Let
us consider one such tuple $(k_{0},m_{0})$ and let us choose $\ell _{0}\in 
\mathbb{Z}$ arbitrarily. First suppose that $(k_{0},m_{0})\notin C_{\ell
_{0}}$, then either%
\begin{equation}
|Q_{k_{0},m_{0}}\cap A_{\ell _{0}}|\leq \frac{|Q_{k_{0},m_{0}}|}{2}\text{%
\quad or\quad }|Q_{k_{0},m_{0}}\cap A_{\ell _{0}+1}|>\frac{|Q_{k_{0},m_{0}}|%
}{2}.  \label{Cond1}
\end{equation}%
Let us assume for the moment that the second condition is satisfied. By
induction on $\ell $ it follows%
\begin{equation}
|Q_{k_{0},m_{0}}\cap A_{\ell +1}|>\frac{|Q_{k_{0},m_{0}}|}{2}\quad \text{for
all\quad }\ell \geq \ell _{0}.  \label{Cond2}
\end{equation}%
Let $D:=\cap _{\ell \geq \ell _{0}}Q_{k_{0},m_{0}}\cap A_{\ell +1}$. The
family $\{Q_{k_{0},m_{0}}\cap A_{\ell }\}_{\ell }$ is a decreasing family of
sets, i.e., $Q_{k_{0},m_{0}}\cap A_{\ell +1}\subset Q_{k_{0},m_{0}}\cap
A_{\ell }$. Therefore, in view of $\mathrm{\eqref{Cond2}}$, the measure of
the set $D$ is larger than or equal to $\frac{|Q_{k_{0},m_{0}}|}{2}$. Hence%
\begin{align}
\big\|\lambda |f_{p,q}(\mathbb{R}^{n},\{\omega _{k}\})\big\|\geq & \Big(%
\int_{Q_{k_{0},m_{0}}\cap A_{\ell }}g^{p}(x)dx\Big)^{\frac{1}{p}}  \label{JF}
\\
\geq & 2^{\ell }\left\vert Q_{k_{0},m_{0}}\cap A_{\ell }\right\vert ^{\frac{1%
}{p}}\geq 2^{\ell }|D|^{\frac{1}{p}},\quad \ell \geq \max (\ell _{0},0). 
\notag
\end{align}%
Since $|D|\geq \frac{|Q_{k_{0},m_{0}}|}{2}>0$, the term on the right-hand
side of $\mathrm{\eqref{JF}}$ is strictly greater than $0$. Now the norm $%
\big\|\lambda |\dot{f}_{p,q}(\mathbb{R}^{n},\{\omega _{k}\})\big\|$ is
finite since $\lambda \in \dot{f}_{p,q}(\mathbb{R}^{n},\{\omega _{k}\})$. In 
$\mathrm{\eqref{JF}}$ letting $\ell $ tends to infinity we get a
contradiction. Hence, we have to turn in $\mathrm{\eqref{Cond1}}$ to the
situation where the first condition is satisfied. We claim that,%
\begin{equation*}
|Q_{k_{0},m_{0}}\cap A_{\ell }|\leq \frac{|Q_{k_{0},m_{0}}|}{2}\quad \text{%
for all\ }\ell \in \mathbb{Z}.
\end{equation*}%
Obviously this yields%
\begin{equation}
|Q_{k_{0},m_{0}}\cap A_{\ell }^{c}|>\frac{|Q_{k_{0},m_{0}}|}{2}\quad \text{%
for all\ }\ell \in \mathbb{Z},  \label{Cond2.1}
\end{equation}%
again\ this claim follows by induction on $\ell $ using $(k_{0},m_{0})\notin
\cup _{\ell \in \mathbb{Z}}C_{\ell }$. Set 
\begin{equation*}
E=\cap _{\ell \geq \max (0,-\ell _{0})}Q_{k_{0},m_{0}}\cap A_{-\ell
}^{c}=\cap _{\ell \geq \max (0,-\ell _{0})}h_{\ell }.
\end{equation*}%
The family $\{h_{\ell }\}_{\ell }$ is a decreasing family of sets, i.e., $%
h_{\ell +1}\subset h_{\ell }$. Therefore, in view of $\mathrm{\eqref{Cond2.1}%
}$, the measure of the set $E$ is larger than or equal to $\frac{%
|Q_{k_{0},m_{0}}|}{2}$. By selecting a point $x\in E$ we obtain%
\begin{equation}
\omega _{k_{0},m_{0}}|\lambda _{k_{0},m_{0}}|\leq g(x)\leq 2^{-\ell }.
\label{JF1}
\end{equation}%
Now, in $\mathrm{\eqref{JF1}}$ if $\ell $ tends to $+\infty $\ then the
claim, namely $\lambda _{k_{0},m_{0}}=0$, follows.

\textit{The choices of }$\lambda _{k,m}^{0}$\textit{\ and }$\lambda
_{k,m}^{1}$. If $(k,m)\notin \cup _{\ell \in \mathbb{Z}}C_{\ell }$, then we
define $\lambda _{k,m}^{0}=\lambda _{k,m}^{1}=0$. If $(k,m)\in C_{\ell }$,
we put\ 
\begin{equation*}
u:=n\big(\tfrac{q}{q_{0}p}-\tfrac{1}{p_{0}}\big)+\tfrac{n}{2}\big(\tfrac{q}{%
q_{0}}-1\big),\quad \vartheta _{k,m}:=w_{k,m}^{\frac{\theta q}{q_{0}}%
}t_{k,m}^{-\frac{\theta q}{q_{1}}}
\end{equation*}%
and 
\begin{equation*}
v:=n\big(\tfrac{q}{q_{1}p}-\tfrac{1}{p_{1}}\big)+\tfrac{n}{2}\big(\tfrac{q}{%
q_{1}}-1\big),\text{\quad }\varepsilon _{k,m}:=t_{k,m}^{\frac{(1-\theta )q}{%
q_{1}}}w_{k,m}^{-\frac{(1-\theta )q}{q_{0}}}.
\end{equation*}%
Let $\delta =\frac{p}{p_{1}}-\frac{q}{q_{1}}$. We put%
\begin{equation*}
\lambda _{k,m}^{0}:=\vartheta _{k,m}2^{ku}2^{\ell \gamma }|\lambda _{k,m}|^{%
\frac{q}{q_{0}}}.
\end{equation*}%
Also, set%
\begin{equation*}
\lambda _{k,m}^{1}:=\varepsilon _{k,m}2^{kv}2^{\ell \delta }|\lambda
_{k,m}|^{\frac{q}{q_{1}}}.
\end{equation*}%
Observe that%
\begin{equation*}
(1-\theta )u+\theta v=0,\quad (1-\theta )\gamma +\theta \delta =0
\end{equation*}%
and%
\begin{equation*}
|\lambda _{k,m}|=\big(\lambda _{k,m}^{0}\big)^{1-\theta }\big(\lambda
_{k,m}^{1}\big)^{\theta },
\end{equation*}%
which holds now for all pairs $(k,m)$.

\textit{Proof of }\eqref{Est-Int}. It will be sufficient to establish the
following two inequalities%
\begin{align}
\big\|\lambda ^{0}|\dot{f}_{p_{0},q_{0}}(\mathbb{R}^{n},\{t_{k}\})\big\|\leq
& c\big\|\lambda |\dot{f}_{p,q}(\mathbb{R}^{n},\{\omega _{k}\})\big\|^{\frac{%
p}{p_{0}}}  \label{Est-Int1} \\
\big\|\lambda ^{1}|\dot{f}_{p_{1},q_{1}}(\mathbb{R}^{n},\{w_{k}\})\big\|\leq
& c\big\|\lambda |\dot{f}_{p,q}(\mathbb{R}^{n},\{\omega _{k}\})\big\|^{\frac{%
p}{p_{1}}}.  \label{Est-Int2}
\end{align}%
Let us prove $\mathrm{\eqref{Est-Int1}}$\textrm{. }Write%
\begin{equation*}
\sum_{k=-\infty }^{\infty }\sum\limits_{m\in \mathbb{Z}%
^{n}}t_{k,m}^{q_{0}}2^{kn(\frac{1}{p_{0}}+\frac{1}{2})q_{0}}(\lambda
_{k,m}^{0})^{q_{0}}\chi _{k,m}=\sum_{\ell =-\infty }^{\infty
}\sum\limits_{(k,m)\in C_{\ell }}t_{k,m}^{q_{0}}2^{kn(\frac{1}{p_{0}}+\frac{1%
}{2})q_{0}}(\lambda _{k,m}^{0})^{q_{0}}\chi _{k,m}=I.
\end{equation*}%
Observe that, 
\begin{align*}
\chi _{k,m}(x)\lesssim & \frac{1}{|Q_{k,m}\cap A_{\ell }|}\int_{Q_{k,m}\cap
A_{\ell }}dy\lesssim \frac{1}{|Q_{k,m}|}\int_{Q_{k,m}\cap A_{\ell }}dy \\
\lesssim & \mathcal{M}(\chi _{Q_{k,m}\cap A_{\ell }})(x),\text{\quad }%
(k,m)\in C_{\ell },
\end{align*}%
where the implicit positive constant not depending on $x$, $\ell ,m$ and $k$%
. We can apply Theorem \ref{Fe-St71} and obtain that 
\begin{equation*}
\big\|I^{\frac{1}{q_{0}}}|L_{p_{0}}(\mathbb{R}^{n})\big\|
\end{equation*}%
is bounded by%
\begin{align*}
& c\Big\|\Big(\sum_{\ell =-\infty }^{\infty }\sum\limits_{(k,m)\in C_{\ell }}%
\Big(\mathcal{M}\big(t_{k,m}2^{kn(\frac{1}{p_{0}}+\frac{1}{2})}\lambda
_{k,m}^{0}\chi _{Q_{k,m}\cap A_{\ell }}\big)\Big)^{q_{0}}\Big)^{\frac{1}{%
q_{0}}}|L_{p_{0}}(\mathbb{R}^{n})\Big\| \\
\lesssim & \Big\|\Big(\sum_{\ell =-\infty }^{\infty }\sum\limits_{(k,m)\in
C_{\ell }}t_{k,m}^{q_{0}}2^{kn(\frac{1}{p_{0}}+\frac{1}{2})q_{0}}(\lambda
_{k,m}^{0})^{q_{0}}\chi _{Q_{k,m}\cap A_{\ell }}\Big)^{\frac{1}{q_{0}}%
}|L_{p_{0}}(\mathbb{R}^{n})\Big\|.
\end{align*}%
We have%
\begin{equation*}
q_{0}-\frac{\theta q}{q_{1}}q_{0}=(1-\theta )q,
\end{equation*}%
hence with application of Lemma \ref{Holder}, 
\begin{equation*}
\vartheta _{k,m}^{q_{0}}t_{k,m}^{q_{0}}=w_{k,m}^{\theta q}t_{k,m}^{(1-\theta
)q}\approx \omega _{k,m}^{q}.
\end{equation*}%
We have%
\begin{equation*}
2^{\ell \gamma }\leq \sum_{\ell =-\infty }^{\infty }\sum\limits_{m\in 
\mathbb{Z}^{n}}\omega _{k,m}^{q}2^{kn(\frac{1}{p}+\frac{1}{2})q}|\lambda
_{k,m}|^{q}\chi _{Q_{k,m}}(x)\Big)^{\frac{\gamma }{q}},
\end{equation*}%
if $x\in A_{\ell }$. Therefore,%
\begin{equation*}
\big\|I^{\frac{1}{q_{0}}}|L_{p_{0}}(\mathbb{R}^{n})\big\|\lesssim \Big\|\Big(%
\sum_{k=-\infty }^{\infty }\sum\limits_{m\in \mathbb{Z}^{n}}2^{kn(\frac{1}{p}%
+\frac{1}{2})q}\omega _{k,m}^{q}|\lambda _{k,m}|^{q}\chi _{Q_{k,m}}\Big)^{%
\frac{1}{q}}|L_{p}(\mathbb{R}^{n})\Big\|^{\frac{p}{p_{0}}}.
\end{equation*}%
Let us prove $\mathrm{\eqref{Est-Int2}}$\textrm{. }We only make some
comments concerning necessary modifications. Again, we write 
\begin{equation*}
\sum_{k=0}^{\infty }\sum\limits_{m\in \mathbb{Z}^{n}}w_{k,m}^{q_{1}}2^{kn(%
\frac{1}{p_{1}}+\frac{1}{2})q_{1}}(\lambda _{k,m}^{1})^{q_{1}}\chi
_{k,m}=\sum_{\ell =-\infty }^{\infty }\sum\limits_{(k,m)\in C_{\ell
}}w_{k,m}^{q_{1}}2^{kn(\frac{1}{p_{1}}+\frac{1}{2})q_{1}}(\lambda
_{k,m}^{1})^{q_{1}}\chi _{k,m}.
\end{equation*}%
Replacing $t_{k,m}$, $q_{0}$, $p_{0}$, $u$ and $\gamma $ by $w_{k,m}$, $%
q_{1} $, $p_{1}$, $v$ and $\delta $, respectively, and $Q_{k,m}\cap $ $%
A_{\ell }$ by $Q_{k,m}\cap A_{\ell +1}^{c}$ and this leads to the desired
inequality in view that $\delta +\frac{q}{q_{1}}=\frac{p}{p_{1}}$.

\noindent \textit{Substep 1.2.2.} We consider the case $\gamma <0$. By an
argument similar to the above with $Q_{k,m}\cap A_{\ell }$, respectively $%
Q_{k,m}\cap A_{\ell +1}^{c}$, replaced by\ $Q_{k,m}\cap A_{\ell +1}^{c}$,
respectively $Q_{k,m}\cap A_{\ell }$.

\noindent \textit{Substep 1.2.3.} We consider the case $\gamma =0$. Then $%
\delta =0$. This case can be easily solved.

\noindent \textit{Step 2.} We deal with the case of $\dot{b}$-spaces. We
prove only the embedding%
\begin{equation*}
\dot{b}_{p,q}(\mathbb{R}^{n},\{\omega _{k}\})\hookrightarrow \big(\dot{b}%
_{p_{0},q_{0}}(\mathbb{R}^{n},\{t_{k}\})\big)^{1-\theta }\big(\dot{b}%
_{p_{1},q_{1}}(\mathbb{R}^{n},\{w_{k}\})\big)^{\theta },
\end{equation*}%
since the opposite embedding follows by H\"{o}lder's inequality. Assume that 
$1\leq q_{0}<q_{1}<\infty $. Let $\mu =\frac{q}{q_{0}}-\frac{p}{p_{0}},\tau =%
\frac{q}{q_{1}}-\frac{p}{p_{1}}$, \ 
\begin{equation*}
u:=\tfrac{n}{2}\big(\tfrac{q}{q_{0}}-1\big),\quad v:=\tfrac{n}{2}\big(\tfrac{%
q}{q_{1}}-1\big),\quad \vartheta _{k,m}:=t_{k,m}^{-\frac{\theta p}{p_{1}}%
}w_{k,m}^{\frac{\theta p}{p_{0}}}
\end{equation*}%
and 
\begin{equation*}
\varepsilon _{k,m}:=t_{k,m}^{\frac{(1-\theta )p}{p_{1}}}w_{k,m}^{-\frac{%
(1-\theta )p}{p_{0}}}.
\end{equation*}%
We put%
\begin{equation*}
\lambda _{k,m}^{0}:=\vartheta _{k,m}2^{ku}|\lambda _{k,m}|^{\frac{p}{p_{0}}}%
\Big(\sum\limits_{h\in \mathbb{Z}^{n}}|\lambda _{k,h}|^{p}t_{k,h}^{p}\Big)^{%
\frac{\mu }{p}},
\end{equation*}%
if $\lambda _{k,m}\neq 0$ and 
\begin{equation*}
\lambda _{k,m}^{0}:=0\quad \text{if}\quad \lambda _{k,m}=0.
\end{equation*}%
Also, set%
\begin{equation*}
\lambda _{k,m}^{1}:=\varepsilon _{k,m}2^{kv}|\lambda _{k,m}|^{\frac{p}{p_{1}}%
}\Big(\sum\limits_{h\in \mathbb{Z}^{n}}|\lambda _{k,h}|^{p}t_{k,h}^{p}\Big)^{%
\frac{\tau }{p}},
\end{equation*}%
if $\lambda _{k,m}\neq 0$ and 
\begin{equation*}
\lambda _{k,m}^{1}:=0\quad \text{if}\quad \lambda _{k,m}=0.
\end{equation*}%
Observe that%
\begin{equation*}
|\lambda _{k,m}|=\big(\lambda _{k,m}^{0}\big)^{1-\theta }\big(\lambda
_{k,m}^{1}\big)^{\theta },
\end{equation*}%
which holds now for all pairs $(k,m)$. We have 
\begin{equation*}
p_{0}-\frac{\theta p}{p_{1}}p_{0}=(1-\theta )p\quad \text{and}\quad p_{1}-%
\frac{(1-\theta )p}{p_{0}}p_{1}=\theta p,
\end{equation*}%
which implies that 
\begin{equation*}
\vartheta _{k,m}^{p_{0}}t_{k,m}^{p_{0}}=w_{k,m}^{\theta p}t_{k,m}^{(1-\theta
)p}\approx \omega _{k,m}^{p}
\end{equation*}%
and%
\begin{equation*}
\varepsilon _{k,m}^{p_{1}}w_{k,m}^{p_{1}}=w_{k,m}^{\theta
p}t_{k,m}^{(1-\theta )p}\approx \omega _{k,m}^{p}.
\end{equation*}%
Simple calculation gives 
\begin{equation*}
\big\|\lambda ^{i}|\dot{b}_{p_{i},q_{i}}(\mathbb{R}^{n},\{t_{k}\})\big\|%
\lesssim \big\|\lambda |\dot{b}_{p,q}(\mathbb{R}^{n},\{\omega _{k}\})\big\|^{%
\frac{q}{q_{i}}},\quad i=0,1.
\end{equation*}%
The proof is complete.
\end{proof}

Next we study the Calder\'{o}n product of $\dot{f}_{p_{0},q_{0}}(\mathbb{R}%
^{n},\{t_{k}\})$ and $\dot{f}_{\infty ,q_{1}}(\mathbb{R}^{n},\{w_{k}\})$.
For $\{t_{k}\}=\{w_{k}\}=\{1\}$, we refer to \cite{FJ90}.

\begin{theorem}
\label{calderon-prod2}Let $0<\theta <1,1\leq p_{0}<\infty $ and $1\leq
q_{0},q_{1}<\infty $. Let $\{t_{k}\}\subset L_{p_{0}}^{\mathrm{loc}}$ and $%
\{w_{k}\}\subset L_{q_{1}}^{\mathrm{loc}}$ be two weight sequences
satisfying $\mathrm{\eqref{Asum1}}$ with $(\sigma _{1},p,j)=((1-\theta )(%
\frac{p_{0}}{1-\theta })^{\prime },p=p_{0},k)$\ and $(\sigma
_{1},p,j)=(\theta (\frac{q_{1}}{\theta })^{\prime },p=q_{1},k)$,
respectively. We put%
\begin{equation*}
\frac{1}{p}:=\frac{1-\theta }{p_{0}},\quad \frac{1}{q}:=\frac{1-\theta }{%
q_{0}}+\frac{\theta }{q_{1}},\quad \omega _{k}:=t_{k}^{1-\theta
}w_{k}^{\theta },\quad k\in \mathbb{Z}.
\end{equation*}%
Then%
\begin{equation*}
\big(\dot{f}_{p_{0},q_{0}}(\mathbb{R}^{n},\{t_{k}\})\big)^{1-\theta }\big(%
\dot{f}_{\infty ,q_{1}}(\mathbb{R}^{n},\{w_{k}\})\big)^{\theta }=\dot{f}%
_{p,q}(\mathbb{R}^{n},\{\omega _{k}\})
\end{equation*}%
holds in the sense of equivalent norms.
\end{theorem}

\begin{proof}
Obviously we can assume that $1<p_{0},p_{1},q_{0},q_{1}<\infty $.

\noindent \textit{Step 1.} In this step we prove the embedding%
\begin{equation*}
\big(\dot{f}_{p_{0},q_{0}}(\mathbb{R}^{n},\{t_{k}\})\big)^{1-\theta }\big(%
\dot{f}_{\infty ,q_{1}}(\mathbb{R}^{n},\{w_{k}\})\big)^{\theta
}\hookrightarrow \dot{f}_{p,q}(\mathbb{R}^{n},\{\omega _{k}\}).
\end{equation*}%
We suppose, that sequences $\lambda :=(\lambda _{k,m})_{j,m}$, $\lambda
^{i}:=(\lambda _{k,m}^{i})_{j,m},i=0,1$, are given and that%
\begin{equation*}
|\lambda _{k,m}|\leq |\lambda _{k,m}^{0}|^{1-\theta }|\lambda
_{k,m}^{1}|^{\theta }
\end{equation*}%
holds for all $k\in \mathbb{Z}$ and $m\in \mathbb{Z}^{n}$. Let $0<\kappa <1$
be such that 
\begin{equation*}
\frac{1}{\kappa p}=\frac{1-\theta }{p_{0}}+\frac{\theta }{q_{1}}.
\end{equation*}%
We set 
\begin{equation}
g:=\Big(\sum_{k=-\infty }^{\infty }\sum\limits_{m\in \mathbb{Z}^{n}}2^{kn(%
\frac{1}{2}+\frac{1}{\kappa p})q}\omega _{k,m,\kappa }^{q}|\lambda
_{k,m}|^{q}\chi _{E_{Q_{k,m}}}\Big)^{\frac{1}{q}},  \label{new-g}
\end{equation}%
with $E_{Q_{k,m}}\subset Q_{k,m}$, $|E_{Q_{k,m}}|>|Q_{k,m}|/2$ and 
\begin{equation*}
\omega _{k,m,\kappa }=\big\|\omega _{k}|L_{\kappa p}(Q_{k,m})\big\|,\quad
k\in \mathbb{Z},m\in \mathbb{Z}^{n}.
\end{equation*}%
By H\"{o}lder's inequality 
\begin{equation*}
\omega _{k,m,\kappa }\leq t_{k,m}^{1-\theta }w_{k,m,q_{1}}^{\theta },\quad
k\in \mathbb{Z},m\in \mathbb{Z}^{n},
\end{equation*}%
with%
\begin{equation*}
t_{k,m}:=\big\|t_{k}|L_{p_{0}}(Q_{k,m})\big\|\quad \text{and}\quad
w_{k,m,q_{1}}=\big\|w_{k}|L_{q_{1}}(Q_{k,m})\big\|,\quad k\in \mathbb{Z}%
,m\in \mathbb{Z}^{n}.
\end{equation*}%
Since 
\begin{align*}
& 2^{kn(\frac{1}{2}+\frac{1}{\kappa p})q}\omega _{k,m,\kappa }^{q}|\lambda
_{k,m}|^{q}\chi _{E_{Q_{k,m}}}(x) \\
\leq & \Big(2^{kn(\frac{1}{2}+\frac{1}{p_{0}})}t_{k,m}|\lambda _{k,m}|\chi
_{E_{Q_{k,m}}}(x)\Big)^{q(1-\theta )}\Big(2^{kn(\frac{1}{2}+\frac{1}{q_{1}}%
)}w_{k,m,q_{1}}|\lambda _{k,m}|\chi _{E_{Q_{k,m}}}(x)\Big)^{q\theta },
\end{align*}%
again, H\"{o}lder's inequality implies that $g$ can be estimated by%
\begin{align*}
& \Big(\sum_{k=-\infty }^{\infty }\sum\limits_{m\in \mathbb{Z}^{n}}2^{kn(%
\frac{1}{2}+\frac{1}{p_{0}})q_{0}}t_{k,m}^{q_{0}}|\lambda
_{k,m}|^{q_{0}}\chi _{E_{Q_{j,m}}}\Big)^{\frac{1-\theta }{q_{0}}} \\
& \times \Big(\sum_{k=-\infty }^{\infty }\sum\limits_{m\in \mathbb{Z}%
^{n}}2^{kn(\frac{1}{2}+\frac{1}{q_{1}})q_{1}}w_{k,m,q_{1}}^{q_{1}}|\lambda
_{k,m}|^{q_{1}}\chi _{E_{Q_{k,m}}}\Big)^{\frac{\theta }{q_{1}}}.
\end{align*}%
We estimate the second factor by its $L^{\infty }$-norm. By using
Proposition \ref{norm-equiv11} we obtain \ 
\begin{equation*}
\big\|\lambda |\dot{f}_{p,q}(\mathbb{R}^{n},\{\omega _{k}\})\big\|\leq \big\|%
\lambda ^{0}|\dot{f}_{p_{0},q_{0}}(\mathbb{R}^{n},\{\omega _{k}\})\big\|%
^{1-\theta }\big\|\lambda ^{1}|\dot{f}_{\infty ,q_{1}},\{w_{k}\})\big\|%
^{\theta }.
\end{equation*}%
\textit{Step 2.} We prove the embedding%
\begin{equation*}
\dot{f}_{p,q}(\mathbb{R}^{n},\{\omega _{k}\})\hookrightarrow \big(\dot{f}%
_{p_{0},q_{0}}(\mathbb{R}^{n},\{t_{k}\})\big)^{1-\theta }\big(\dot{f}%
_{\infty ,q_{1}}(\mathbb{R}^{n},\{w_{k}\})\big)^{\theta }.
\end{equation*}%
Let the sequence $\lambda \in \dot{f}_{p,q}(\mathbb{R}^{n},\{\omega _{k}\})$
be given. We have to find sequences $\lambda ^{0}$ and $\lambda ^{1}$ such
that 
\begin{equation*}
|\lambda _{k,m}|\leq M|\lambda _{k,m}^{0}|^{1-\theta }|\lambda
_{k,m}^{1}|^{\theta }
\end{equation*}%
for every $k\in \mathbb{Z}$, $m\in \mathbb{Z}^{n}$ and 
\begin{equation}
\big\|\lambda ^{0}|\dot{f}_{p_{0},q_{0}}(\mathbb{R}^{n},\{t_{k}\})\big\|%
^{1-\theta }\big\|\lambda ^{1}|\dot{f}_{\infty ,q_{1}}(\mathbb{R}%
^{n},\{w_{k}\})\big\|^{\theta }\leq c\big\|\lambda |\dot{f}_{p,q}(\mathbb{R}%
^{n},\{\omega _{k}\})\big\|.  \label{Est-Int.New}
\end{equation}%
Let 
\begin{equation*}
\gamma =\frac{p}{p_{0}}-\frac{q}{q_{0}}\quad \text{and}\quad \delta =-\frac{q%
}{q_{1}}.
\end{equation*}%
We have to subdivide Step 2 into the two substeps:

\noindent \textit{Substep 2.1.} We consider the case $\gamma >0$. Set 
\begin{equation*}
A_{\ell }:=\{x\in \mathbb{R}^{n}:g(x)>2^{\ell }\},
\end{equation*}%
with $\ell \in \mathbb{Z}$ and $g$ as in $\mathrm{\eqref{new-g}}$. Obviously 
$A_{\ell +1}\subset A_{\ell }$, with $\ell \in \mathbb{Z}$. Now we introduce
a (partial) decomposition of $\mathbb{Z}\times \mathbb{Z}^{n}$ by taking%
\begin{equation*}
C_{\ell }:=\{(k,m):|Q_{k,m}\cap A_{\ell }|>\frac{|Q_{k,m}|}{2}\text{\quad
and\quad }|Q_{k,m}\cap A_{\ell +1}|\leq \frac{|Q_{k,m}|}{2}\},\quad \ell \in 
\mathbb{Z}.
\end{equation*}%
The sets $C_{\ell }$ are pairwise disjoint, i.e., $C_{\ell }\cap
C_{v}=\emptyset $ if $\ell \neq v$. As in Theorem \ref{calderon-prod2} we
can prove that $\lambda _{k,m}=0$ holds for all tuples $(k,m)\notin \cup
_{\ell \in \mathbb{Z}}C_{\ell }$. If $(k,m)\notin \cup _{\ell \in \mathbb{Z}%
}C_{\ell }$, then we define $\lambda _{k,m}^{0}=\lambda _{k,m}^{1}=0$. If $%
(k,m)\in C_{\ell }$, we put\ 
\begin{equation*}
u:=n\big(\tfrac{q}{q_{0}\kappa p}-\tfrac{1}{p_{0}}\big)+\tfrac{n}{2}\big(%
\tfrac{q}{q_{0}}-1\big),\quad \vartheta _{k,m}:=t_{k,m}^{-\frac{\theta q}{%
q_{1}}}w_{k,m,q_{1}}^{\frac{\theta q}{q_{0}}}
\end{equation*}%
and 
\begin{equation*}
v:=n\big(\tfrac{q}{q_{1}\kappa p}-\tfrac{1}{q_{1}}\big)+\tfrac{n}{2}\big(%
\tfrac{q}{q_{1}}-1\big),\text{\quad }\varepsilon _{k,m}:=t_{k,m}^{\frac{%
(1-\theta )q}{q_{1}}}w_{k,m,q_{1}}^{-\frac{(1-\theta )q}{q_{0}}}.
\end{equation*}%
We put%
\begin{equation*}
\lambda _{k,m}^{0}:=\vartheta _{k,m}2^{ku}2^{\ell \gamma }|\lambda _{k,m}|^{%
\frac{q}{q_{0}}}\quad \text{and}\quad \lambda _{k,m}^{1}:=\varepsilon
_{k,m}2^{kv}2^{\ell \delta }|\lambda _{k,m}|^{\frac{q}{q_{1}}}.
\end{equation*}%
Observe that%
\begin{equation*}
(1-\theta )u+\theta v=0,\quad (1-\theta )\gamma +\theta \delta =0
\end{equation*}%
and%
\begin{equation*}
|\lambda _{k,m}|=\big(\lambda _{k,m}^{0}\big)^{1-\theta }\big(\lambda
_{k,m}^{1}\big)^{\theta },
\end{equation*}%
which holds now for all pairs $(k,m)$. To prove $\mathrm{\eqref{Est-Int.New}}
$, it will be sufficient to establish the following two inequalities%
\begin{align}
\big\|\lambda ^{0}|\dot{f}_{p_{0},q_{0}}(\mathbb{R}^{n},\{t_{k}\})\big\|\leq
& c\big\|\lambda |\dot{f}_{p,q}(\mathbb{R}^{n},\{\omega _{k}\})\big\|^{\frac{%
p}{p_{0}}}  \label{Est-Int1.New} \\
\big\|\lambda ^{1}|\dot{f}_{\infty ,q_{1}}(\mathbb{R}^{n},\{w_{k}\})\big\|%
\leq & c.  \label{Est-Int2.New}
\end{align}

\textit{Proof of }\eqref{Est-Int1.New}\textrm{. }Write%
\begin{equation*}
\sum_{k=-\infty }^{\infty }\sum\limits_{m\in \mathbb{Z}%
^{n}}t_{k,m}^{q_{0}}2^{kn(\frac{1}{p_{0}}+\frac{1}{2})q_{0}}(\lambda
_{k,m}^{0})^{q_{0}}\chi _{k,m}=\sum_{\ell =-\infty }^{\infty
}\sum\limits_{(k,m)\in C_{\ell }}t_{k,m}^{q_{0}}2^{kn(\frac{1}{p_{0}}+\frac{1%
}{2})q_{0}}(\lambda _{k,m}^{0})^{q_{0}}\chi _{k,m}=I.
\end{equation*}%
Observe that, 
\begin{align*}
\chi _{k,m}(x)\lesssim & \frac{1}{|Q_{k,m}\cap A_{\ell }|}\int_{Q_{k,m}\cap
A_{\ell }}dy\lesssim \frac{1}{|Q_{k,m}|}\int_{Q_{k,m}\cap A_{\ell }}dy \\
\lesssim & \mathcal{M}(\chi _{Q_{k,m}\cap A_{\ell }})(x),\text{\quad }%
(k,m)\in C_{\ell },
\end{align*}%
where the implicit positive constant not depending on $x$, $\ell ,m$ and $k$%
. We can apply Theorem \ref{Fe-St71} and obtain that 
\begin{equation*}
\big\|I^{\frac{1}{q_{0}}}|L_{p_{0}}(\mathbb{R}^{n})\big\|
\end{equation*}%
is bounded by%
\begin{align*}
& c\Big\|\Big(\sum_{\ell =-\infty }^{\infty }\sum\limits_{(k,m)\in C_{\ell }}%
\Big(\mathcal{M}\big(t_{k,m}2^{kn(\frac{1}{p_{0}}+\frac{1}{2})}\lambda
_{k,m}^{0}\chi _{Q_{k,m}\cap A_{\ell }}\big)\Big)^{q_{0}}\Big)^{\frac{1}{%
q_{0}}}|L_{p_{0}}(\mathbb{R}^{n})\Big\| \\
\lesssim & \Big\|\Big(\sum_{\ell =-\infty }^{\infty }\sum\limits_{(k,m)\in
C_{\ell }}t_{k,m}^{q_{0}}2^{kn(\frac{1}{p_{0}}+\frac{1}{2})q_{0}}(\lambda
_{k,m}^{0})^{q_{0}}\chi _{Q_{k,m}\cap A_{\ell }}\Big)^{\frac{1}{q_{0}}%
}|L_{p_{0}}(\mathbb{R}^{n})\Big\|.
\end{align*}%
We have $q_{0}-\frac{\theta q}{q_{1}}q_{0}=(1-\theta )q$, which together
with Lemma \ref{Holder} yields 
\begin{equation*}
\vartheta _{k,m}^{q_{0}}t_{k,m}^{q_{0}}=w_{k,m,q_{1}}^{\theta
q}t_{k,m}^{(1-\theta )q}\approx \omega _{k,m,\kappa }^{q}.
\end{equation*}%
In addition, we have%
\begin{equation*}
2^{\ell \gamma }\leq \sum_{\ell =-\infty }^{\infty }\sum\limits_{m\in 
\mathbb{Z}^{n}}\omega _{k,m,\kappa }^{q}2^{kn(\frac{1}{\kappa p}+\frac{1}{2}%
)q}|\lambda _{k,m}|^{q}\chi _{Q_{k,m}}(x)\Big)^{\frac{\gamma }{q}},
\end{equation*}%
if $x\in A_{\ell }$. Therefore,%
\begin{align*}
\big\|I^{\frac{1}{q_{0}}}|L_{p_{0}}(\mathbb{R}^{n})\big\|\lesssim & \Big\|%
\Big(\sum_{k=-\infty }^{\infty }\sum\limits_{m\in \mathbb{Z}^{n}}2^{kn(\frac{%
1}{\kappa p}+\frac{1}{2})q}\omega _{k,m,\kappa }^{q}|\lambda _{k,m}|^{q}\chi
_{Q_{k,m}}\Big)^{\frac{1}{q}}|L_{p}(\mathbb{R}^{n})\Big\|^{\frac{p}{p_{0}}}
\\
\lesssim & \big\|\lambda |\dot{f}_{p,q}(\mathbb{R}^{n},\{\omega _{k}\})\big\|%
^{\frac{p}{p_{0}}}.
\end{align*}

\textit{Proof of }\eqref{Est-Int2.New}\textrm{. }By Proposition \ref%
{norm-equiv11} with $E_{Q_{k,m}}^{\ell }=Q_{k,m}\cap A_{\ell +1}^{c}$,%
\begin{equation*}
\big\|\lambda ^{1}|\dot{f}_{\infty ,q_{1}}(\mathbb{R}^{n},\{w_{k}\})\big\|%
\lesssim \Big\|\Big(\sum_{\ell =-\infty }^{\infty }\sum\limits_{(k,m)\in
C_{\ell }}2^{kn(\frac{1}{q_{1}}+\frac{1}{2})q_{1}}w_{k,m,q_{1}}^{q_{1}}|%
\lambda _{k,m}^{1}|^{q_{1}}\chi _{E_{Q_{k,m}}^{\ell }}\Big)^{1/q_{1}}\Big\|%
_{\infty }.
\end{equation*}%
Observe that%
\begin{equation*}
w_{k,m,q_{1}}\lambda _{k,m}^{1}\lesssim \omega _{k,m,\kappa }^{\frac{q}{q_{1}%
}}2^{\ell \delta +kv}|\lambda _{k,m}|^{\frac{q}{q_{1}}}\leq \omega
_{k,m,\kappa }^{\frac{q}{q_{1}}}2^{kv}g^{-\frac{q}{q_{1}}}(x)|\lambda
_{k,m}|^{\frac{q}{q_{1}}}
\end{equation*}%
for any $x\in E_{Q_{k,m}}^{\ell }$ and%
\begin{equation*}
v+\frac{n}{q_{1}}+\frac{n}{2}=\frac{nq}{q_{1}}(\frac{1}{\kappa p}+\frac{1}{2}%
).
\end{equation*}%
Therefore, 
\begin{equation*}
\big\|\lambda ^{1}|\dot{f}_{\infty ,q_{1}}(\mathbb{R}^{n},\{w_{k}\})\big\|%
\lesssim 1.
\end{equation*}%
\textit{Substep 1.2.2.} We consider the case $\gamma <0$. By an argument
similar to the above with $Q_{k,m}\cap A_{\ell }$ replaced by\ $Q_{k,m}\cap
A_{\ell +1}^{c}$.

Hence, we complete the proof.
\end{proof}

There are nice connections between complex interpolation spaces and the
corresponding Calder\'{o}n product, see the original paper of Calder\'{o}n 
\cite{Ca64}. Suppose \ that $X_{0}$ and $X_{1}$ are Banach lattices on
measure space $\left( \mathcal{M},\mu \right) $, and let 
\begin{equation}
X=X_{0}^{1-\theta }\cdot X_{1}^{\theta }  \label{calderon1}
\end{equation}%
for some $0<\theta <1$. Suppose that $X$ hus the property%
\begin{equation*}
f\in X,\text{ \ }\left\vert f_{n}\left( x\right) \right\vert \leq \left\vert
f\left( x\right) \right\vert \text{, }\mu \text{-a.e., \ and \ }%
\lim_{n\rightarrow \infty }f_{n}=f,\mu \text{-}a.e.\Longrightarrow
\lim_{n\rightarrow \infty }\left\Vert f_{n}\right\Vert _{X}=\left\Vert
f\right\Vert _{X},
\end{equation*}%
Calder\'{o}n \cite[p. 125]{Ca64} then shows that 
\begin{equation}
X_{0}^{1-\theta }\cdot X_{1}^{\theta }=[X_{0},X_{1}]_{\theta }.
\label{calderon}
\end{equation}

Let $1\leq p<\infty $, $1\leq q<\infty $ and $\{t_{k}\}$ be a $p$-admissible
weight sequence. In $\mathrm{\eqref{calderon1}}$ assume that $X=\dot{a}%
_{p,q}(\mathbb{R}^{n},\{t_{k}\})$ for some Banach lattices $X_{0}$ and $%
X_{1} $. By dominated convergence theorem $X$ satisfies the Calder\'{o}n's
result $\mathrm{\eqref{calderon}}$. Using Theorem \ref{calderon-prod1} and
Calder\'{o}n's result we get the following theorem.

\begin{theorem}
\label{Ind-resol2}Let $0<\theta <1,1\leq p_{0}<\infty $, $1\leq p_{1}<\infty 
$ and $1\leq q_{0},q_{1}<\infty $. Let $\{t_{k}\}\subset L_{p_{0}}^{\mathrm{%
loc}}$ and $\{w_{k}\}\subset L_{p_{1}}^{\mathrm{loc}}$ be two weight\
sequences satisfying $\mathrm{\eqref{Asum1}}$ with $(\sigma
_{1},p,j)=((1-\theta )(\frac{p_{0}}{1-\theta })^{\prime },p=p_{0},k)$\ and $%
(\sigma _{1},p,j)=(\theta (\frac{p_{1}}{\theta })^{\prime },p=p_{1},k)$,
respectively. We put%
\begin{equation*}
\frac{1}{q}:=\frac{1-\theta }{q_{0}}+\frac{\theta }{q_{1}}\text{\quad
and\quad }\omega _{k}:=t_{k}^{1-\theta }w_{k}^{\theta },\quad k\in \mathbb{Z}%
.
\end{equation*}%
Let $\frac{1}{p}:=\frac{1-\theta }{p_{0}}+\frac{\theta }{p_{1}}$. Then%
\begin{equation*}
\lbrack \dot{a}_{p_{0},q_{0}}(\mathbb{R}^{n},\{t_{k}\}),\dot{a}%
_{p_{1},q_{1}}(\mathbb{R}^{n},\{w_{k}\})]_{\theta }=\dot{a}_{p,q}(\mathbb{R}%
^{n},\{\omega _{k}\})
\end{equation*}%
holds in the sense of equivalent norms.
\end{theorem}

By Theorems \ref{phi-tran} and \ref{Ind-resol2}, we easily obtain the
following result.

\begin{theorem}
Let $\alpha =(\alpha _{1},\alpha _{2})\in \mathbb{R}^{2},\beta =(\beta
_{1},\beta _{2})\in \mathbb{R}^{2},0<\theta <1,1\leq p_{0}<\infty $, $1\leq
p_{1}<\infty $ and $1\leq q_{0},q_{1}<\infty $. Let $\{t_{k}\}\in \dot{X}%
_{\alpha ,\sigma ,p_{0}}$ be a $p_{0}$-admissible weight sequence with $%
\sigma =((1-\theta )(\frac{p_{0}}{1-\theta })^{\prime },\sigma _{2}\geq
p_{0})$. Let $\{w_{k}\}\in \dot{X}_{\beta ,\sigma ,p_{1}}$ be a $p_{1}$%
-admissible weight sequence with $\sigma =(\theta (\frac{p_{1}}{\theta }%
)^{\prime },\sigma _{2}\geq p_{1})$. We put%
\begin{equation*}
\frac{1}{p}:=\frac{1-\theta }{p_{0}}+\frac{\theta }{p_{1}},\text{\quad }%
\frac{1}{q}:=\frac{1-\theta }{q_{0}}+\frac{\theta }{q_{1}}\text{\quad
and\quad }\omega _{k}:=t_{k}^{1-\varrho }w_{k}^{\varrho },\quad k\in \mathbb{%
Z}.
\end{equation*}%
Then%
\begin{equation*}
\lbrack \dot{A}_{p_{0},q_{0}}(\mathbb{R}^{n},\{t_{k}\}),\dot{A}%
_{p_{1},q_{1}}(\mathbb{R}^{n},\{w_{k}\})]_{\theta }=\dot{A}_{p,q}(\mathbb{R}%
^{n},\{\omega _{k}\})
\end{equation*}%
holds in the sense of equivalent norms.
\end{theorem}

From Theorem \ref{calderon-prod2} and Calder\'{o}n's result we obtain the
following statement.

\begin{theorem}
\label{calderon-prod2 copy(1)}Let $0<\theta <1,1\leq p_{0}<\infty $ and $%
1\leq q_{0},q_{1}<\infty $. Let $\{t_{k}\}\subset L_{p_{0}}^{\mathrm{loc}}$
and $\{w_{k}\}\subset L_{p_{1}}^{\mathrm{loc}}$ be two weight sequences
satisfying $\mathrm{\eqref{Asum1}}$ with $(\sigma _{1},p,j)=((1-\theta )(%
\frac{p_{0}}{1-\theta })^{\prime },p=p_{0},k)$\ and $(\sigma
_{1},p,j)=(\theta (\frac{q_{1}}{\theta })^{\prime },p=q_{1},k)$,
respectively. We put%
\begin{equation*}
\frac{1}{p}:=\frac{1-\theta }{p_{0}},\quad \frac{1}{q}:=\frac{1-\theta }{%
q_{0}}+\frac{\theta }{q_{1}},\quad \omega _{k}:=t_{k}^{1-\theta
}w_{k}^{\theta },\quad k\in \mathbb{Z}.
\end{equation*}%
Then%
\begin{equation*}
\lbrack \dot{f}_{p_{0},q_{0}}(\mathbb{R}^{n},\{t_{k}\}),\dot{f}_{\infty
,q_{1}}(\mathbb{R}^{n},\{t_{k}\})]_{\theta }=\dot{f}_{p,q}(\mathbb{R}%
^{n},\{\omega _{k}\})
\end{equation*}%
holds in the sense of equivalent norms.
\end{theorem}

Theorems \ref{phi-tran}, \ref{phi-tran2}, \ref{Ind-resol2} and \ref%
{calderon-prod2 copy(1)} yield the following result.

\begin{theorem}
\label{p=infinity}Let $\alpha =(\alpha _{1},\alpha _{2})\in \mathbb{R}%
^{2},\beta =(\beta _{1},\beta _{2})\in \mathbb{R}^{2},0<\theta <1$, $1\leq
p_{1}<\infty $ and $1\leq q_{0},q_{1}<\infty $. Let $\{t_{k}\}\in \dot{X}%
_{\alpha ,\sigma ,p_{0}}$ be a $p_{0}$-admissible weight sequence with $%
\sigma =((1-\theta )(\frac{p_{0}}{1-\theta })^{\prime },\sigma _{2}\geq
p_{0})$. Let $\{w_{k}\}\in \dot{X}_{\beta ,\sigma ,q_{1}}$ be a $p_{1}$%
-admissible weight sequence with $\sigma =(\theta (\frac{q_{1}}{\theta }%
)^{\prime },\sigma _{2}\geq q_{1})$. We put%
\begin{equation*}
\frac{1}{p}:=\frac{1-\theta }{p_{0}},\text{\quad }\frac{1}{q}:=\frac{%
1-\theta }{q_{0}}+\frac{\theta }{q_{1}}\text{\quad and\quad }\omega
_{k}:=t_{k}^{1-\theta }w_{k}^{\theta },,\quad k\in \mathbb{Z}.
\end{equation*}%
Then%
\begin{equation*}
\lbrack \dot{F}_{p_{0},q_{0}}(\mathbb{R}^{n},\{t_{k}\}),\dot{F}_{\infty
,q_{1}}(\mathbb{R}^{n},\{w_{k}\})]_{\theta }=\dot{F}_{p,q}(\mathbb{R}%
^{n},\{\omega _{k}\})
\end{equation*}%
holds in the sense of equivalent norms.
\end{theorem}

\begin{remark}
The methods of \cite{SSV13} are capable of dealing with the spaces $\dot{B}%
_{p,q}(\mathbb{R}^{n},\{t_{k}\})$ and $\dot{F}_{p,q}(\mathbb{R}%
^{n},\{t_{k}\})$ with the smoothness $t_{k}$ independent of $k,k\in \mathbb{Z%
}$, but they do not apply in the case of spaces of variable smoothness as in
this paper.
\end{remark}

\begin{corollary}
Let $s_{0},s_{1}\in \mathbb{R},0<\theta <1,1\leq p_{0}<\infty $ and $1\leq
q_{0},q_{1}<\infty $. Let $\omega _{0}^{p_{0}}\in A_{\frac{p_{0}}{1-\theta }%
}(\mathbb{R}^{n})$ and $\omega _{1}^{q_{1}}\in A_{\frac{q_{1}}{\theta }}(%
\mathbb{R}^{n})$,%
\begin{equation*}
\frac{1}{p}:=\frac{1-\theta }{p_{0}},\text{\quad }\frac{1}{q}:=\frac{%
1-\theta }{q_{0}}+\frac{\theta }{q_{1}}\quad \text{and}\quad s=(1-\theta
)s_{0}+\theta s_{1}.
\end{equation*}%
Then%
\begin{equation*}
\lbrack \dot{F}_{p_{0},q_{0}}^{s_{0}}(\mathbb{R}^{n},\omega _{0}),\dot{F}%
_{\infty ,q_{1}}^{s_{1}}(\mathbb{R}^{n},\omega _{1})]_{\theta }=\dot{F}%
_{p,q}^{s}(\mathbb{R}^{n},\omega _{0}^{1-\theta }\omega _{1}^{\theta })
\end{equation*}%
holds in the sense of equivalent norms.
\end{corollary}

\begin{proof}
The proof is easily followed by Theorem \ref{p=infinity}.
\end{proof}

\begin{remark}
All our results are easily generalized to the inhomogeneous Triebel-Lizorkin
spaces $F_{p,q}(\mathbb{R}^{n},\{t_{k}\})$ and the inhomogeneous Besov
spaces $B_{p,q}(\mathbb{R}^{n},\{t_{k}\})$.
\end{remark}


\begin{thebibliography}{99}
\bibitem{AJ80} K. Andersen, R. John, Weighted inequalities for vector-valued
maximal functions and singular integrals, Studia Math. 69 (1980), 19--31.

\bibitem{BL76} J. Bergh, J. L\"{o}fstr\"{o}m, Interpolation spaces. An
introduction, Springer, Berlin, 1976.

\bibitem{B03} O.V. Besov, Equivalent normings of spaces of functions of
variable smoothness, Function spaces, approximation, and differential
equations, A collection of papers dedicated to the 70th birthday of Oleg
Vladimorovich Besov, a corresponding member of the Russian Academy of
Sciences, Tr. May. Inst. Steklova, vol. 243, Nauka, Moscow 2003, pp. 87--95;
English transl. in Proc. Steklov Inst. Math. 243 (2003), 80--88.

\bibitem{B05} O.V. Besov, Interpolation, embedding, and extension of spaces
of functions of variable smoothness, Investigations in the theory of
functions and differential equations, A collection of papers dedicated to
the 100th birthday of academician Sergei Mikhailovich Nikol'skii, Tr. Mat.
Inst. Steklova, vol. 248, Nauka, Moscow 2005, pp. 52--63; English transl.
Proc. Steklov Inst. Math. 248 (2005), 47--58.

\bibitem{BoHo06} M. Bownik, K.-P. Ho, Atomic and molecular decompositions of
anisotropic Triebel-Lizorkin spaces, Trans. Amer. Math. Soc. 358 (2006),
1469--1510.

\bibitem{M07} M. Bownik, Anisotropic Triebel-Lizorkin spaces with doubling
measures, J. Geom. Anal. 17 (2007), 337--424.

\bibitem{M08} M. Bownik, Duality and interpolation of anisotropic
Triebel-Lizorkin spaces, Math. Z. 259 (2008), 131--169.

\bibitem{Bui82} H.Q. Bui, Weighted Besov and Triebel spaces: interpolation
by the real method, Hiroshima Math. J. 12 (1982), 581--605.

\bibitem{Ca64} A.P. Calder\'{o}n, Intermediate spaces and interpolation, the
complex method, Studia Math. 24 (1964), 113--190.

\bibitem{CF88} F. Cobos, D.L. Fernandez, Hardy-Sobolev spaces and Besov
spaces with a function parameter, In M. Cwikel and J. Peetre, editors,
Function spaces and applications, volume 1302 of Lect. Notes Math., pages
158--170. Proc. US-Swed. Seminar held in Lund, June, 1986, Springer, 1988.

\bibitem{D7} D. Drihem, Spline representations of Lizorkin-Triebel spaces
with general weights, submitted

\bibitem{D20} D. Drihem, Besov spaces\ with general weights. J. Math. Study.
56, No. 1 (2023), 18-92.

\bibitem{D20.1} D. Drihem, Triebel-Lizorkin spaces with general weights.
Adv. Oper. Theory 8, 5 (2023). https://doi.org/10.1007/s43036-022-00230-0

\bibitem{D20.2} D. Drihem, Duality of function spaces with general weights.
Preprint.

\bibitem{Du01} J. Duoandikoetxea, Fourier Analysis, Grad. Studies in Math.
29, American Mathematical Society, Providence, RI, 2001.

\bibitem{ET96} D. Edmunds, H. Triebel, Spectral theory for isotropic fractal
drums. C.R. Acad. Sci. Paris 326, s\'{e}rie I (1998), 1269--1274.

\bibitem{ET99} D. Edmunds, H. Triebel, Eigenfrequencies of isotropic fractal
drums. Oper. Theory: Adv. \& Appl. 110 (1999), 81--102.

\bibitem{FL06} W. Farkas, H.-G. Leopold, Characterisations of function
spaces of generalised smoothness, Annali di Mat. Pura Appl. 185 (2006),
1--62.

\bibitem{FeSt71} C. Fefferman, E.M. Stein, Some maximal inequalities, Amer.
J. Math. 93 (1971), 107--115.

\bibitem{FJ86} M. Frazier, B. Jawerth, Decomposition of Besov spaces,
Indiana Univ. Math. J. 34 (1985), 777--799.

\bibitem{FJ90} M. Frazier, B. Jawerth,\ A discrete transform and
decomposition of distribution spaces, J. Funct. Anal. 93 (1990), 34--170.

\bibitem{FrJaWe01} M. Frazier, B. Jawerth, G. Weiss, Littlewood-Paley Theory
and the Study of Function Spaces. CBMS Regional Conference Series in
Mathematics, vol. 79. Published for the Conference Board of the Mathematical
Sciences, Washington, DC. American Mathematical Society, Providence (1991).

\bibitem{GR85} J. Garc\.{\i}a-Cuerva, J. L. Rubio de Francia. Weighted Norm
Inequalities and Related Topics, North-Holland Mathematics Studies, 116.
Notas de Matem\'{a}tica [Mathematical Notes], 104. North-Holland Publishing
Co., Amsterdam, 1985.

\bibitem{Go79} M.L. Goldman, A description of the traces of some function
spaces, Trudy Mat. Inst. Steklov. 150 (1979), 99--127, (Russian) English
transl.: Proc. Steklov Inst. Math. 1981, no. 4 (150).

\bibitem{Go83} M.L. Goldman, A method of coverings for describing general
spaces of Besov type, Trudy Mat. Inst. Steklov. 156 (1980), 47--81,
(Russian) English transl.: Proc. Steklov Inst. Math. 1983, no. 2 (156).

\bibitem{L. Graf14} L. Grafakos, Classical Fourier Analysis, Third Edition,
Graduate Texts in Math., no 249, Springer, New York, 2014.

\bibitem{J77} B. Jawerth, Some observations on Besov and Lizorkin-Triebel
spaces,\ Math. Stand. 40 (1977), 94--104.

\bibitem{KMM07} N. Kalton, S. Mayboroda, M. Mitrea, Interpolation of
Hardy-Sobolev-Besov-Triebel-Lizorkin spaces and applications to problems in
partial differential equations, Contemporary Math., 445 (2007), 121--177.

\bibitem{Ka83} G.A. Kalyabin, Description of functions from classes of
Besov-Lizorkin-Triebel type, Tr. Mat. Inst. Steklova. 156 (1980), 82--109,
(Russian) English translation: Proc. Steklov Inst. Math. 1983, (156).

\bibitem{Kl87} G.A. Kalyabin, P.I. Lizorkin, Spaces of functions of
generalized smoothness, Math. Nachr. 133 (1987), 7--32.

\bibitem{Kok78} V.M. Kokilashvili, Maximum inequalities and multipliers in
weighted Lizorkin-Triebel spaces, Dokl. Akad. Nauk SSSR, 239, (1978) 42--45.
Russian.

\bibitem{Ky03} G. Kyriazis, Decomposition systems for function spaces,
Studia Math. 157(2) (2003), 133--169.

\bibitem{Mo01} S.D. Moura, Function spaces of generalised smoothness, Diss.
Math. \textbf{398} (2001), 88 pp..

\bibitem{Mu72} B. Muckenhoupt, Weighted norm inequalities for the Hardy
maximal function, Trans. Amer. Math. Soc. 165 (1972), 207--226.

\bibitem{HaS08} D.D. Haroske, S.D. Moura, Continuity envelopes and sharp
embeddings in spaces of generalized smoothness, J. Funct. Anal. 254(6)
(2008), 1487--1521.

\bibitem{Ry01} V.S. Rychkov, Littlewood-Paley theory and function spaces
with $A_{p}^{\mathsf{loc}}$\ -weights, Math. Nachr. 224(1) (2001), 145-180.

\bibitem{SchTr87} H.-J. Schmeisser, H. Triebel, Topics in Fourier analysis
and function spaces, Wiley, Chichester, 1987.

\bibitem{Sch982} T. Schott, Function spaces with exponential weights, II.
Math. Nachr. 196 (1998), 231--250.

\bibitem{SSV13} W. Sickel, L. Skrzypczak and J. Vyb\'{\i}ral, Complex
interpolation of weighted Besov- and Lizorkin-Triebel spaces, Acta. Math.
Sci, 30(8) (2014), 1297--1323.

\bibitem{St93} E.M. Stein. Harmonic analysis: real-variable methods,
orthogonality, and oscillatory integrals, Princeton Mathematical Series, 43.
Monographs in Harmonic Analysis, III. Princeton University Press, Princeton,
NJ, 1993.

\bibitem{ST89} J.-O. Str\"{o}mberg, A. Torchinsky, Weighted Hardy spaces,
Lecture Notes in Mathematics, 1381, Springer-Verlag, Berlin, 1989.

\bibitem{T78} H. Triebel, Interpolation theory, function spaces,
differential operators, Verlag der Wissenschaften, Berlin, 1978.

\bibitem{T1} H. Triebel, Theory of function spaces, Birkh\"{a}user Verlag,
Basel, 1983.

\bibitem{T2} H. Triebel, Theory of function spaces II, Birkh\"{a}user
Verlag, Basel, 1992.

\bibitem{Ty14} A.I. Tyulenev, Description of traces of functions in the
Sobolev space with a Muckenhoupt weight, Function spaces and related
questions in analysis, A collection of papers dedicated to the 80th birthday
of Oleg Vladimorovich Besov, a corresponding member of the Russian Academy
of Sciences, Tr. Mat. Inst. Steklova, vol. 284, MAIK, Moscow 2014, pp.
288-303; English transl. in Proc. Steklov Inst. Math. 284 (2014) 280--295.

\bibitem{Ty15} A.I. Tyulenev, Some new function spaces of variable
smoothness, Sbornik Mathematics. 206 (2015), 849--891.

\bibitem{Ty-N-L} A.I. Tyulenev, Besov-type spaces of variable smoothness on
rough domains, Nonlinear Anal. 145 (2016), 176--198.

\bibitem{Ty-151} A.I. Tyulenev, On various approaches to Besov-type spaces
of variable smoothness, J. Math. Anal. Appl. 451 (2017), 371--392.

\bibitem{Wo12} A. Wojciechowska, Multidimensional wavelet bases in Besov and
Lizorkin-Triebel spaces, PhD-thesis, Adam Mickiewicz University Pozna\'{n},
Pozna\'{n}, 2012.

\bibitem{YY2} D. Yang, W. Yuan, New Besov-type spaces and
Triebel-Lizorkin-type spaces including $Q$ spaces, Math. Z. 265 (2010),
451--480.

\bibitem{YYZ13} D. Yang, W. Yuan, C. Zhuo, Complex interpolation on
Besov-type and Triebel-Lizorkin-type spaces. Anal Appl (Singap)., (2013),
11: 1350021, 45pp.
\end{thebibliography}
\end{document}